\documentclass{amsart}

\usepackage{enumerate,mathdots,mathrsfs,verbatim,fullpage,amssymb}

\numberwithin{equation}{section}

\newtheorem{prop}{Proposition}[section]
\newtheorem{theorem}[prop]{Theorem}
\newtheorem{cor}[prop]{Corollary}
\newtheorem{lemma}[prop]{Lemma}

\theoremstyle{definition}
\newtheorem{defn}[prop]{Definition}
\newtheorem{example}[prop]{Example}

\theoremstyle{remark}
\newtheorem{rem}[prop]{Remark}

\newcommand{\A}{\mathscr{A}}

\newcommand{\C}{\mathbb{C}}

\newcommand{\F}{\mathcal{F}}
\newcommand{\G}{\mathcal{G}}
\renewcommand{\H}{\mathcal{H}}
\newcommand{\HS}{\mathcal{HS}}
\newcommand{\K}{\mathcal{K}}
\newcommand{\M}{\mathcal{M}}
\newcommand{\N}{\mathbb{N}}

\newcommand{\R}{\mathbb{R}}

\newcommand{\Norm}[1]{\left\Vert #1 \right\Vert}

\newcommand{\caret}{\char`\^}

\renewcommand{\subset}{\subseteq}
\renewcommand{\supset}{\supseteq}

\DeclareMathOperator{\tr}{tr}

\DeclareMathOperator{\mult}{mult}
\DeclareMathOperator{\card}{card}
\DeclareMathOperator{\spn}{span}
\DeclareMathOperator{\ran}{ran}
\DeclareMathOperator{\rank}{rank}

\title{Frames generated by compact group actions}
\author{Joseph W.\ Iverson}
\address{Department of Mathematics, University of Oregon, Eugene, OR 97403--1222, USA}
\email{iverson@uoregon.edu}
\date{\today}


\begin{document}

\begin{abstract}

Let $K$ be a compact group, and let $\rho$ be a representation of $K$ on a Hilbert space $\H_\rho$. We classify invariant subspaces of $\H_\rho$ in terms of range functions, and investigate frames of the form $\{\rho(\xi) f_i\}_{\xi \in K, i \in I}$. This is done first in the setting of translation invariance, where $K$ is contained in a larger group $G$ and $\rho$ is left translation on $\H_\rho = L^2(G)$. For this case, our analysis relies on a new, operator-valued version of the Zak transform. For more general representations, we develop a calculational system known as a \emph{bracket} to analyze representation structures and frames with a single generator. Several applications are explored. Then we turn our attention to frames with multiple generators, giving a duality theorem that encapsulates much of the existing research on frames generated by finite groups, as well as classical duality of frames and Riesz sequences.

\end{abstract}

\maketitle

This paper is an investigation of the interplay between frame theory and representations of compact groups. Broadly speaking, we are interested in two related questions about a unitary representation $\rho$ of a compact group $K$: 
\begin{center}
(1) What are the invariant subspaces of $\rho$?
\end{center}
and 
\begin{center}
(2) For which families $\A$ of vectors in the representation space is  the orbit 

$\{ \rho(\xi) f : \xi \in K, f \in \A\}$ a frame?
\end{center}
These questions are related in the following way. Often, the vectors $\{\rho(\xi) f: \xi \in K, f\in \A\}$ do not span the entire representation space, in which case they can only form a frame for their closed linear span. That span is precisely the invariant subspace generated by $\A$. In the most general setting, we will use frame theory to answer the first question, and representation theory to answer the second. 

The overarching theme of this paper is that frame theory and representation theory share deep connections. By this we mean much more than the prominence of reproducing systems associated with group actions. As we will see, many of the standard tools of frame theory give vital information about the structure of  representations. In Section 4, for instance, we develop an analogue of the \emph{bracket map}, which found its first use in the study of multiresolution analysis \cite{JM}. It turns out that the bracket carries information about the isotypical components of a representation and the multiplicities of irreducibles, and in many cases can be used to test a purported cyclic vector. In the final section, we give a complete description of the invariant subspaces of an arbitrary representation of a compact group, and explain how to use one irreducible decomposition to classify all such decompositions. The main tool for both of these applications is essentially the analysis operator. We did not go out looking for these results, but stumbled into them where they lay directly in the path of our investigation of frame properties. 

Many of the prototypical examples of frames, including wavelets and Gabor systems, are associated with group actions \cite{DGM}. Frames of the form described in (2), which occupy the full orbit of a vector family, are particularly nice. We will call these objects \emph{group frames}, and say the vectors in $\A$ are \emph{generators}. Examples range from the continuous wavelet transform, which is associated with an action of the $ax+b$ group \cite{GrMo,K}, to harmonic frames, which come from actions of finite abelian groups \cite{VW2}. The reproducing properties of group frames are often greatly simplified by the aid of the representation. When there is a single generator $f$, for instance, the frame operator $S$ lies in the commutant of $\rho$. This means that when one wants to reproduce a vector $g$ with the formula 
\[ g = \int_K \langle g, S^{-1} \rho(\xi) f \rangle \rho(\xi) f\, d\xi, \]
there is no need to compute $S^{-1} \rho(\xi) f$ for every $\xi \in K$. It suffices to compute $S^{-1} f$ and then observe that
\[  \langle g, S^{-1} \rho(\xi) f \rangle = \langle \rho(\xi^{-1}) g, S^{-1} f \rangle. \] 
Group frames are made using the natural symmetries of the representation space, and as a consequence they often combine utility, beauty, and simplicity. In Example \ref{ex:PermFrm}, for instance, we explain how to make a unit norm tight frame for $\C^n$ consisting of $n!$ vectors, just by permuting the entries of a single vector. 

\medskip

The paper is split into three parts, each of which can be read more or less independently from the others. The first part, Sections \ref{sec:Zak}--\ref{sec:tranFrm}, investigates questions (1) and (2) for actions of compact groups by translation. Let $G$ be a second countable locally compact group, and let $K \subset G$ be a compact subgroup. The purpose of these sections is to describe the structure of closed subspaces of $L^2(G)$ which are invariant under left translation by $K$. We call these spaces \emph{$K$-invariant}. Our first major development occurs in Section \ref{sec:Zak}, where we introduce an operator-valued analogue of the Zak transform, generalizing a classical construction of Weil \cite{W2,W} and Gelfand \cite{G}. It forms the basis for much of our subsequent analysis. In Section \ref{sec:ranTran}, we make our first mention of \emph{range functions}, which make several appearances throughout the paper. We use range functions to classify $K$-invariant subspaces of $L^2(G)$, and explore this correspondence in depth. This line of thinking comes to a culmination in Section \ref{sec:tranFrm}, where we give precise conditions for a family of functions in $L^2(G)$ to generate a frame via left translation by $K$.

The second part of the paper, Sections \ref{sec:brack} and \ref{sec:brackApp}, describes a symbolic calculus for the analysis of representations of compact groups. We introduce an operator-valued version of the bracket map first developed for the study of principle shift-invariant spaces by Jia and Micchelli \cite{JM}, and subsequently generalized for actions of locally compact abelian (LCA) groups by Weiss and his collaborators \cite{HSWW2}, then by a variety of authors in other settings \cite{BHM,BHP,BHP3}. Our main result, Theorem \ref{thm:brackFrm}, gives the frame properties of the orbit of a cyclic vector in terms of the eigenvalues of the bracket. We develop basic properties of the bracket in Section \ref{sec:brack}. Several of these show the bracket carries vital information about the structure of the representation itself. Section \ref{sec:brackApp} contains a host of applications: classification of group frames with a single generator, block diagonalization of the Gramian operator, disjointness properties, and several new examples of frames, including a generalization of harmonic frames for nonabelian groups.

The third part, Section \ref{sec:multGen}, is dedicated to group frames with multiple generators. Here we mimic the program of Sections \ref{sec:Zak}--\ref{sec:tranFrm} for an arbitrary representation $\rho$ of a compact group $K$, assuming only that we know how to decompose $\rho$ as a direct sum of irreducible subrepresentations. We classify the invariant subspaces of $\rho$ using range functions and a sort of analysis operator, then describe every possible decomposition of the representation space as a direct sum of irreducible invariant subspaces. The capstone of this section, and the culmination of the entire paper, is the duality result in Theorem \ref{thm:multFrm2}. Three decades since the start of the wavelet revolution, we still do not have a satisfactory answer for a very simple question:
\begin{center}
(2') Given a locally compact group $G$, a unitary representation $\pi \colon G \to U(\H_\pi)$, and a family of vectors $\A \subset \H_\pi$, under what circumstances is the orbit $\{ \pi(x) f : x \in G, f \in \A\}$ a frame?
\end{center}
Theorem \ref{thm:multFrm2} answers this question for representations of compact groups using a simple duality statement. Our result unifies classical duality of frames and Riesz sequences with, among other things, the pioneering work of Vale and Waldron \cite{VW2,VW,VW3}, and the well-known result that the orbit of a nonzero vector under an irreducible representation of $K$ always forms a tight frame. We hope that this theorem, and many of the other ideas in this paper, will give some clues for subsequent research on group frames.

\medskip

\section{The Zak transform of a compact subgroup} \label{sec:Zak}

In Sections \ref{sec:Zak} -- \ref{sec:tranFrm}, $G$ is a second countable locally compact group (not necessarily abelian), and $K\subset G$ is a compact subgroup. Our main result is the existence of an operator-valued Zak transform on $L^2(G)$ that treats left translation by $K$ in a manner similar to the Fourier transform on $L^2(K)$. This operator will form the basis for our classification of $K$-invariant subspaces of $L^2(G)$ in Section \ref{sec:ranTran}, and for our analysis of frames formed by $K$-translates in Section \ref{sec:tranFrm}. 

The reader may consult \cite{F,HR2} for background on compact groups and their representations. We record a few of the basics here. Throughout the paper, we normalize Haar measure on $K$ so that $|K| = 1$. The left and right translates of $f\colon K \to \C$ by $\xi \in K$ are denoted $L_\xi f$ and $R_\xi f$, respectively. That is,
\[ (L_\xi f)(\eta) = f(\xi^{-1} \eta), \quad (R_\xi f)(\eta) = f(\eta \xi) \qquad (\eta \in K). \]
We give $L^2(K)$ the usual convolution and involution, namely
\[ (f * g)(\xi) = \int_K f(\eta) g(\eta^{-1} \xi)\, d\eta \qquad (f,g\in L^2(K);\ \xi \in K) \]
and
\[ (f^*)(\xi) = \overline{f(\xi^{-1})} \qquad (f \in L^2(K),\ \xi \in K). \]
These operations make $L^2(K)$ a Banach $*$-algebra. 

The dual object of $K$ is $\hat{K}$; it has one representative of each equivalence class of irreducible unitary representations of $K$. Each $\pi \in \hat{K}$ acts on a finite dimensional space, which we denote $\H_\pi$. Its dimension is $d_\pi = \dim \H_\pi$. The \emph{Fourier transform} of $f\in L^2(K)$ evaluated at $\pi \in \hat{K}$ is the operator
\[ \hat{f}(\pi) = \int_K f(\xi) \pi(\xi^{-1})\, d\xi \in B(\H_\pi), \]
where the integral is to be interpreted in the weak sense. For our purposes, the utility of the Fourier transform lies in the formulae
\begin{equation} \label{eq:FourTrans}
(L_\xi f)\caret(\pi) = \hat{f}(\pi) \pi(\xi^{-1}), \quad (R_\xi f)\caret(\pi) = \pi(\xi) \hat{f}(\pi) \qquad (f \in L^2(K),\ \xi \in K,\ \pi \in \hat{K})
\end{equation}
and
\begin{equation}\label{eq:FourConv}
(f^*)\caret(\pi) = \hat{f}(\pi)^*, \quad (f * g)\caret(\pi) = \hat{g}(\pi) \hat{f}(\pi) \qquad (f,g \in L^2(K);\ \pi \in \hat{K}).
\end{equation}
If $B(\H_\pi)$ is treated as a Hilbert space with inner product $\langle A, B \rangle = d_\pi \langle A,B \rangle_{\HS} = d_\pi \tr(B^*A)$, the Fourier transform may be viewed as a unitary
\[ \F \colon L^2(K) \to \bigoplus_{\pi \in \hat{K}} B(\H_\pi), \qquad \F f = (\hat{f}(\pi))_{\pi \in \hat{K}}. \]
This is called \emph{Plancherel's Theorem}. When an orthonormal basis $e_1^\pi,\dotsc,e_{d_\pi}^\pi \in \H_\pi$ is chosen for each $\pi \in \hat{K}$, we define the \emph{matrix elements} $\pi_{i,j} \in C(K)$ by
\[ \pi_{i,j}(\xi) = \langle \pi(\xi) e_j^\pi, e_i^\pi \rangle \qquad (\pi \in \hat{K};\ \xi \in K;\ i,j = 1,\dotsc,d_\pi). \]
In other words, the matrix for $\pi(\xi)$ with respect to the chosen basis is $(\pi_{i,j}(\xi))_{i,j=1}^{d_\pi}$. For $f \in L^2(K)$, the $(i,j)$-entry of the matrix for $\hat{f}(\pi)$ over this basis is 
\[ \hat{f}(\pi)_{i,j} =  \int_K f(\xi) \overline{\pi_{i,j}(\xi)}\, d\xi \qquad (f \in L^2(K);\ \pi \in \hat{K};\ i,j=1,\dotsc,d_\pi). \]
The \emph{contragredient} to $\pi\in \hat{K}$ is the representation $\overline{\pi}$ on $\H_\pi$ with matrix elements
\[ \overline{\pi}_{i,j}(\xi) = \overline{\pi_{i,j}(\xi)} \qquad (\xi \in K;\, i,j=1,\dotsc,d_\pi). \]
The contragredient of an irreducible representation is also irreducible. The \emph{Peter-Weyl Theorem} asserts that 
\[ \{ \sqrt{d_\pi} \pi_{i,j} : \pi \in \hat{K},\, i,j = 1,\dotsc,d_\pi\} \]
is an orthonormal basis for $L^2(K)$. In particular,
\begin{equation}\label{eq:Planch}
\Norm{f}_{L^2(K)}^2 = \sum_{\pi \in \hat{K}} \sum_{i,j=1}^{d_\pi} d_\pi | \hat{f}(\pi)_{i,j}|^2 = \sum_{\pi \in \hat{K}} d_\pi \Norm{ \hat{f}(\pi) }_{\HS}^2 \qquad (f \in L^2(K)).
\end{equation}

\medskip

Let $K\backslash G$ be the quotient space of \emph{right} cosets of $K$ in $G$. A \emph{cross section} of $K\backslash G$ in $G$ is a map $\tau \colon K\backslash G \to G$ that selects a representative of each coset. In other words, $\tau(Kx) \in Kx$ for every $Kx \in K\backslash G$. By a classic result of Feldman and Greenleaf \cite{FG}, there is a Borel cross section $\tau \colon K\backslash G \to G$ which maps compact subsets of $K\backslash G$ to sets with compact closure in $G$. Fix such a cross section, and let $T \colon K \times K \backslash G \to G$ be the bijection
\begin{equation} \label{eq:measIsom}
T(\xi, Kx) = \xi\cdot \tau(Kx) \qquad (\xi \in K,\ Kx \in K\backslash G).
\end{equation}
By \cite[Theorem 3.6]{I}, $K\backslash G$ admits a unique regular Borel measure with respect to which $T$ is a measure space isomorphism. We shall always have this measure in mind when we treat $K\backslash G$ as a measure space.

Given a function $f\colon G \to \C$ and a coset $Kx \in K \backslash G$, we will denote $f_{Kx} \colon K \to \C$ for the function given by
\begin{equation} \label{eq:cosetFunct}
f_{Kx}(\xi) = f(\xi\cdot \tau(Kx)) \qquad (\xi \in K).
\end{equation}
Intuitively, we are treating the coset $Kx$ like a copy of $K$ itself, with the chosen representative $\tau(Kx)$ taking the role of the identity element. In this sense, $f_{Kx}$ is just the restriction of $f$ to $Kx$. Obviously,
\begin{equation} \label{eq:tranRes}
(L_\xi f)_{Kx} = L_\xi(f_{Kx}) \qquad (\xi \in K,\ Kx \in K\backslash G).
\end{equation}

\smallskip

\begin{theorem} \label{thm:Zak}
There is a unitary
\[ Z \colon L^2(G) \to \bigoplus_{\pi \in \hat{K}} B(\H_\pi, L^2(K\backslash G ; \H_\pi) ) \]
given by
\begin{equation} \label{eq:ZakFour}
[(Zf)(\pi) u ](Kx) = (f_{Kx})\caret (\pi) u \qquad (f \in L^2(G),\ \pi \in \hat{K},\ u \in \H_\pi,\ Kx \in K\backslash G).
\end{equation}
Here $B(\H_\pi, L^2(K\backslash G ; \H_\pi) )$ is treated as a Hilbert space with inner product $\langle A, B \rangle = d_\pi \tr(B^*A)$, and the direct sum is that of Hilbert spaces.

For $f \in L^2(G)$, $\xi \in K$, and $\pi \in \hat{K}$, the unitary $Z$ satisfies
\begin{equation}\label{eq:ZakTrans}
[Z (L_\xi f)](\pi) = (Zf)(\pi)\, \pi(\xi^{-1}).
\end{equation}
\end{theorem}

We call $Z$ the \emph{Zak transform} for the pair $(G,K)$. 

\begin{proof}
The measure space isomorphism $T\colon K \times K\backslash G \to G$ induces a unitary $U \colon L^2(G) \to L^2(K\times K\backslash G)$, namely
\[ (Uf)(\xi,Kx) = f(\xi \cdot \tau(Kx)) = f_{Kx}(\xi) \qquad (f \in L^2(G),\ \xi \in K,\ Kx \in K\backslash G). \]
Follow this with the canonical unitary $V \colon L^2(K\times K\backslash G) \to L^2(K) \otimes L^2(K\backslash G)$, and then apply 
\[ \F_K \otimes \text{id} \colon L^2(K) \otimes L^2(K\backslash G) \to [ \bigoplus_{\pi \in \hat{K}} B(\H_\pi) ] \otimes L^2(K\backslash G). \]
Finally, make the natural identifications
\[ [ \bigoplus_{\pi \in \hat{K}} B(\H_\pi) ] \otimes L^2(K\backslash G) \cong \bigoplus_{\pi \in \hat{K}} [ B(\H_\pi) \otimes L^2(K\backslash G) ] \cong \bigoplus_{\pi \in \hat{K}} B(\H_\pi, \H_\pi \otimes L^2(K\backslash G) ) \]
\[ \cong \bigoplus_{\pi \in \hat{K}} B(\H_\pi, L^2(K\backslash G; \H_\pi)). \]
The resulting composition is $Z$. The translation identity \eqref{eq:ZakTrans} follows directly from \eqref{eq:tranRes}, \eqref{eq:ZakFour}, and the corresponding identity for the Fourier transform \eqref{eq:FourTrans}.
\end{proof}

\begin{rem} \label{rem:ZakFour}
In the extreme case where $K$ is all of $G$, the quotient $K\backslash G$ consists of a single point, and we can interpret $L^2(K\backslash G; \H_\pi)$ as simply being $\H_\pi$. Then the Zak transform reduces to the usual Fourier transform on $L^2(K)$, as long as the cross section $\tau$ chooses the identity element as the representative of the (single) coset of $K$ in $G$.
\end{rem}

\sloppy
In general, the choice of cross-section $\tau$ is noncanonical, and the operator $Z$ depends on this choice. Nonetheless, Zak transforms associated with different cross-sections are easily related. Suppose that $\tau' \colon K\backslash G \to G$ is another cross-section with the required properties. For each $Kx \in K\backslash G$, there is an element $\eta_{Kx} \in K$ such that $\tau'(Kx) = \eta_{Kx} \tau(Kx)$. Denoting $Z_\tau$ and $Z_{\tau'}$ for the versions of the Zak transform obtained using $\tau$ and $\tau'$, respectively, we obtain the following formula from \eqref{eq:ZakFour} and \eqref{eq:FourTrans}:
\[ [(Z_{\tau'} f)(\pi) u](Kx) = \pi(\eta_{Kx}) [(Z_\tau f)(\pi) u](Kx) \qquad (f \in L^2(G),\ \pi \in \hat{K},\ u \in \H_\pi,\ Kx \in K\backslash G). \]
In other words, $Z_{\tau'}$ can be obtained from $Z_\tau$ by applying post composition with $\pi(\eta_{Kx})$ at each point $Kx \in K\backslash G$ and in every coordinate $\pi \in \hat{K}$.

As with the usual Fourier transform on $L^2(K)$, there is another, basis-dependent version of the Zak transform that sometimes makes computation more convenient. When an orthonormal basis $e_1^{\pi}, \dotsc, e_{d_\pi}^\pi$ is chosen for $\H_\pi$, the space $B(\H_\pi, L^2(K\backslash G; \H_\pi) )$ can be identified with $M_{d_\pi}(L^2(K\backslash G))$ by mapping the operator $A$ to the matrix whose $(i,j)$-entry is the function
\[ Kx \mapsto \langle (A e_j^\pi)(Kx), e_i^\pi \rangle \qquad (Kx \in K\backslash G). \]
Under this identification, the inner product on $M_{d_\pi}(L^2(K\backslash G))$ corresponding to the one in the definition of the Zak transform is given by
\[ \langle M, N \rangle = d_\pi \sum_{i,j=1}^{d_\pi} \langle M_{i,j}, N_{i,j} \rangle \qquad (M,N \in M_{d_\pi}(L^2(K\backslash G)) ). \]
When this identification is made for each $\pi \in \hat{K}$, the Zak transform becomes a unitary
\[ \tilde{Z} \colon L^2(G) \to \bigoplus_{\pi \in \hat{K}} M_{d_\pi}(L^2(K\backslash G)). \]
The translation formula \eqref{eq:ZakTrans} then becomes
\begin{equation} \label{eq:ZakTransBas}
[\tilde Z (L_\xi f)](\pi) = (\tilde{Z} f)(\pi) \cdot (\pi_{i,j}(\xi^{-1}) )_{i,j=1}^{d_\pi} \qquad (f \in L^2(G),\ \xi \in K,\ \pi \in \hat{K}),
\end{equation}
where the vector- and scalar-valued matrices multiply using the usual formula for matrix multiplication. For $f \in L^2(G)$ and $\pi \in \hat{K}$, the $(i,j)$-entry of $(\tilde Z f)(\pi)$ is the function in $L^2(K\backslash G)$ given by
\begin{equation} \label{eq:ZakBas}
Kx \mapsto \int_K f(\xi \tau(Kx)) \pi_{i,j}(\xi^{-1})\, d\xi \qquad (Kx \in K\backslash G).
\end{equation}

For example, when $K$ is a compact \emph{abelian} group, each irreducible representation $\pi \in \hat{K}$ has dimension 1. Thus $M_{d_\pi}( L^2(K\backslash G) )$ can be identified with $L^2(K\backslash G)$, and if we reinterpret the direct sum, we may view the Zak transform as a unitary
\[ \tilde{\tilde{Z}} \colon L^2(G) \to \ell^2(\hat{K}; L^2(K\backslash G) ) \]
given by
\[ [(\tilde{\tilde{Z}} f)(\alpha)](Kx) = \int_K f(\xi \tau(Kx)) \overline{\alpha(\xi)}\, d\xi \qquad (f \in L^2(G),\ \alpha \in \hat{K},\ Kx \in K \backslash G). \]
This agrees with the notion of Zak transform for an abelian subgroup described by the author in \cite{I}. If $G$ and $K$ are both abelian, this definition is equivalent to the original notion of Zak transform as described by Weil in \cite[p. 164--165]{W}. That version of the Zak transform has a very long history in harmonic analysis. We refer the reader to \cite{HSWW} for a brief survey.

\medskip

\section{Range functions and translation invariance} \label{sec:ranTran}

A closed subspace $V \subset L^2(G)$ will be called \emph{$K$-invariant} if $L_\xi f \in V$ whenever $f \in V$ and $\xi \in K$. In this section, we apply the Zak transform to classify the $K$-invariant subspaces of $L^2(G)$ in terms of \emph{range functions}.

\begin{defn}
Let $X$ be an indexing set, and let $\mathscr{H} = \{\H(x)\}_{x\in X}$ be a family of Hilbert spaces. A \emph{range function} in $\mathscr{H}$ is a mapping
\[ J \colon X \to \bigcup_{x\in X} \{\text{closed subspaces of } \H(x)\} \]
such that $J(x) \subset \H(x)$ for each $x \in X$. In other words, it is a choice of closed subspace $J(x) \subset \H(x)$ for each $x \in X$.
\end{defn}

If $J$ is a range function in $\{L^2(K\backslash G; \H_\pi)\}_{\pi \in \hat{K}}$, we define
\[ V_J = \{ f \in L^2(G) : \text{for all $\pi \in \hat{K}$, the range of $(Zf)(\pi)$ is contained in $J(\pi)$}\}. \]
In terms of the Zak transform,
\begin{equation} \label{eq:ZVJ}
Z(V_J) = \bigoplus_{\pi \in \hat{K}} B(\H_\pi, J(\pi)),
\end{equation}
where we consider $B(\H_\pi, J(\pi) )$ to be a closed subspace of $B(\H_\pi, L^2(K\backslash G; \H_\pi) )$. The translation identity \eqref{eq:ZakTrans} for the Zak transform shows that $V_J$ is $K$-invariant. Remarkably, every $K$-invariant subspace of $L^2(G)$ takes this form.

\begin{theorem} \label{thm:ranTran}
The mapping $J \mapsto V_J$ is a bijection between range functions in $\{L^2(K\backslash G; \H_\pi)\}_{\pi \in \hat{K}}$ and $K$-invariant subspaces of $L^2(G)$.
\end{theorem}

A basis-dependent version of this theorem runs as follows. Choose orthonormal bases for each of the spaces $\H_\pi$, $\pi \in \hat{K}$, and let
\[ \tilde{Z} \colon L^2(G) \to \bigoplus_{\pi \in \hat{K}} M_{d_\pi}( L^2(K\backslash G) ) \]
be the resulting basis-dependent Zak transform. For each $\pi \in \hat{K}$, we will think of the columns of $M_{d_\pi}( L^2(K\backslash G) )$ as elements of $L^2(K\backslash G)^{\oplus d_\pi}$, the direct sum of $d_\pi$ copies of $L^2(K\backslash G)$. Given a range function $J$ in $\{ L^2(K\backslash G)^{\oplus d_\pi} \}_{\pi \in \hat{K}}$, let
\[ \tilde{V}_J = \{ f \in L^2(G) : \text{for all $\pi \in \hat{K}$, the columns of $(\tilde{Z}f)(\pi)$ lie in $J(\pi)$}\}. \]
Then $J \mapsto \tilde{V}_J$ is a bijection between range functions in $\{ L^2(K\backslash G)^{\oplus d_\pi}\}_{\pi \in \hat{K}}$ and $K$-invariant subspaces of $L^2(G)$.

Range functions have a long history in the theory of translation invariance. Helson \cite{He} and Srinivasan \cite{S} seem to have the first results in this area. Their work was released at approximately the same time, and each cites the other, so it is not clear who deserves credit for this line of research. The idea of applying a Fourier-like transform and classifying invariant subspaces in terms of range functions has since been applied by a host of researchers in a variety of settings \cite{ACHKM,ACP,ACP2,B,B2,BR,CP,CP2,CMO,BDR,He2,KR}. Recently, the author \cite{I} and Hern{\'a}ndez, et al.\ \cite{BHP2} independently applied a version of the Zak transform to classify translation invariance by an abelian subgroup. The present theorem extends these results to the setting of compact groups. 

We emphasize the novelty of applying this technique with a nonabelian subgroup. Of the results mentioned above, only Currey, et al.\ \cite{CMO} treats a noncommutative case.\footnote{At least one other group of researchers has shown interest in generalizing the shift-invariance results of \cite{B} to the compact nonabelian setting. An attempt at a classification theorem appears in \cite{RK}.} The theory of translation invariance in the nonabelian setting is only in its beginning stages. We hope that by first understanding the case of compact groups, where the representation theory is comparatively simple, we can help point a direction for understanding more general locally compact nonabelian groups.

\medskip

The proof of Theorem \ref{thm:ranTran} relies on a standard decomposition of actions of compact groups. If $\rho\colon K \to U(\H_\rho)$ is a unitary representation of $K$, then the \emph{isotypical component} of $\pi \in \hat{K}$ in $\rho$ is the invariant subspace $\mathcal{M}_\pi \subset \H_\rho$ spanned by all subspaces of $\H_\rho$ on which $\rho$ is unitarily equivalent to $\pi$. Then 
\begin{equation} \label{eq:dirSumIso}
\H_\rho = \bigoplus_{\pi \in \hat{K}} \mathcal{M}_\pi.
\end{equation}
Moreover, each $\mathcal{M}_\pi$ decomposes as a direct sum of irreducible subspaces on which $\rho$ is equivalent to $\pi$. If $\mult(\pi,\rho)$ is the multiplicity of $\pi$ in $\rho$, it follows that
\begin{equation} \label{eq:multDimIso}
\dim \mathcal{M}_\pi = d_\pi \cdot \mult(\pi,\rho) \qquad (\pi \in \hat{K}).
\end{equation}
See \cite[\S 5.1]{F}. 

Given an invariant subspace $V \subset \H_\rho$, we will write $\rho^V$ for the subrepresentation of $\rho$ on $V$. The following can be deduced easily from \cite[Theorem 27.44]{HR2}.

\begin{lemma} \label{lem:isoCom}
Let $\rho \colon K \to U(\H_\rho)$ be a unitary representation of $K$, with isotypical components $\mathcal{M}_\pi \subset \H_\rho$ for $\pi\in \hat{K}$.
\begin{enumerate}[(i)]
\item For each $\pi \in \hat{K}$, let $E_\pi \subset \H_\rho$ be the closed linear span of some invariant subspaces on which $\rho$ is equivalent to $\pi$. If $\H_\rho = \bigoplus_{\pi \in \hat{K}} E_\pi$, then $E_\pi = \mathcal{M}_\pi$ for every $\pi \in \hat{K}$.

\item If $V \subset \H_\rho$ is an invariant subspace, then $V \cap \mathcal{M}_\pi$ is the isotypical component of $\pi \in \hat{K}$ in $\rho^V$.
\end{enumerate}
\end{lemma}

The next lemma follows from Schur's Lemma and the Double Commutant Theorem for von Neumann algebras.

\begin{lemma} \label{lem:repSpn}
Let $\pi \in \hat{K}$. Then $B(\H_\pi) = \spn\{ \pi(\xi) : \xi \in K\}$.
\end{lemma}

\medskip

\noindent \emph{Proof of Theorem \ref{thm:ranTran}.}
Since $Z$ is unitary, \eqref{eq:ZVJ} shows that the mapping $J \mapsto V_J$ is injective. We need only prove that every $K$-invariant subspace $V\subset L^2(G)$ arises as such a $V_J$. To do this, we will first show that $V$ decomposes as a direct sum of simpler pieces, and then we will leverage the Zak transform's translation property \eqref{eq:ZakTrans} on each piece.

Let $\rho$ be the action of $K$ on $L^2(G)$ by left translation.  For each $\pi \in \hat{K}$, let
\[ M_\pi = \{f \in L^2(G) : (Zf)(\sigma) = 0 \text{ for }\sigma \neq \pi\}. \]
We claim that $M_\pi$ is the isotypical component of $\overline{\pi}$ in $\rho$. Fix an orthonormal basis $e_1^\pi,\dotsc,e_{d_\pi}^\pi \in \H_\pi$. For each $\pi \in \hat{K}$, $i=1,\dotsc,d_\pi$, and nonzero $F \in L^2(K\backslash G; \H_\pi)$, we define $F_{\pi,i} \in L^2(G)$ by
\[ (ZF_{\pi,i})(\sigma)e_j^\sigma = \begin{cases}
d_\pi^{-1/2} \Norm{F}^{-1} \cdot F, & \text{if }\sigma = \pi \text{ and }i=j \\
0, & \text{otherwise}
\end{cases} \qquad (\sigma \in \hat{K};\ j = 1,\dotsc,d_\sigma)\]
Then $\langle F_{\pi,i}, F_{\pi,j} \rangle = \langle Z F_{\pi,i}, Z F_{\pi,} \rangle = \delta_{i,j}$, and one can check that
\[ L_\xi F_{\pi,j} = \sum_{i=1}^{d_\pi} \overline{\pi}_{i,j}(\xi)\cdot F_{\pi,i}. \]
Hence $\spn\{F_{\pi,i} : i=1,\dotsc,d_\pi\}$ is a $K$-invariant subspace of $M_\pi$ on which $\rho$ is equivalent to $\overline{\pi}$. Moreover,
\[ M_\pi = \overline{\spn}\{ F_{\pi,i} : F\in L^2(K\backslash G; \H_\pi), i = 1,\dotsc, d_\pi\}. \]
The claim follows from Lemma \ref{lem:isoCom}(i).

Now let $V \subset L^2(G)$ be a $K$-invariant subspace. By Lemma \ref{lem:isoCom}(ii),
\[ V = \bigoplus_{\pi \in \hat{K}} V \cap M_\pi. \]
Since $(Zf)(\sigma)=0$ for $f\in V\cap M_\pi$ and $\sigma \neq \pi$, we may view $W_\pi := Z(V \cap M_\pi)$ as a closed subspace of $B(\H_\pi, L^2(K\backslash G; \H_\pi))$. Let
\[ J(\pi) = \overline{\spn}\{ Au : A \in W_\pi, u \in \H_\pi \} \subset L^2(K\backslash G; \H_\pi). \]
Clearly $W_\pi \subset B(\H_\pi, J(\pi) )$. If we can upgrade this inclusion to equality, we will be able to conclude that
\[ Z(V) = \bigoplus_{\pi \in \hat{K}} Z(V\cap M_\pi) = \bigoplus_{\pi \in \hat{K}} B(\H_\pi, J(\pi) ), \]
and the proof will be complete.

Fix any $A \in B(\H_\pi, J(\pi) )$. We want to show that $A \in W_\pi$. A moment's thought shows that $A$ is the sum of operators in $B(\H_\pi, J(\pi) )$ whose kernels have codimension one. It is enough to show that each of those operators belongs to $W_\pi$. We may therefore assume that there is a unit norm vector $u \in \H_\pi$ such that $Av = 0$ for all $v \perp u$. Let $\epsilon > 0$ be arbitrary. Since $\ran A \subset J(\pi)$, we can find operators $B_1,\dotsc,B_n \in W_\pi$ and nonzero vectors $v_1,\dotsc,v_n \in \H_\pi$ such that
\[ \Norm{ Au - \sum_{j=1}^n B_j v_j }^2 < \epsilon. \]
We are going to produce an operator $B \in W_\pi$ with $B u = \sum_{j=1}^n B_j v_j$ and $B v = 0$ for $v \perp u$.

Here is the key step. Since $V \cap M_\pi$ is invariant under left translation by $K$, the identity \eqref{eq:ZakTrans} shows that $W_\pi = Z(V\cap M_\pi)$ is invariant under right multiplication by $\pi(\xi^{-1})$ for each $\xi \in K$. Therefore $W_\pi$ is invariant under right multiplication by $B(\H_\pi) = \spn\{\pi(\xi^{-1}) : \xi \in K\}$. In particular, we can precompose each $B_j \in W_\pi$ with another operator in $B(\H_\pi)$ to make $B_j' \in W_\pi$ satisfying $B_j' u = B_j v_j$ and $B_j' v = 0$ for all $v \perp u$. Then $B:= B_1' + \dotsb + B_n'$ belongs to $W_\pi$, and 
\[ \Norm{ A - B }^2 = d_\pi \Norm{ Au - Bu }^2 < d_\pi \epsilon. \]
Since $W_\pi$ is closed and $\epsilon >0$ was arbitrary, we conclude that $A \in W_\pi$. Therefore,
\[ Z(V \cap M_\pi) = W_\pi = B(\H_\pi, J(\pi) ), \]
as desired. \hfill \qed

\medskip

The preceding proof contained a fact that is useful in its own right.

\begin{prop} \label{prop:isoComp}
Let $J$ be a range function in $\{L^2(K\backslash G; \H_\pi)\}_{\pi \in \hat{K}}$, and let $\rho_J$ be the representation of $K$ on $V_J$ given by left translation. Then the isotypical component of $\pi \in \hat{K}$ in $\rho_J$ is 
\[ \mathcal{M}_\pi = \{ f \in V_J : (Zf)(\sigma) = 0 \text{ for }\sigma \neq \overline{\pi} \}. \]
In particular,
\begin{equation} \label{eq:multRan}
\mult(\pi, \rho_J) = \dim J(\overline{\pi}).
\end{equation}
and
\begin{equation} \label{eq:dimRan}
\dim V_J = \sum_{\pi \in \hat{K}} d_\pi\cdot \dim J(\overline{\pi}).
\end{equation}
\end{prop}

\begin{proof}
That $\mathcal{M}_\pi$ is the isotypical component of $\pi$ in $\rho_J$ was proven above. To see \eqref{eq:multRan}, simply observe that
\[ \dim \mathcal{M}_\pi = \dim Z \mathcal{M}_\pi = d_\pi \cdot \dim J(\pi) \]
and apply \eqref{eq:multDimIso}. Then \eqref{eq:dimRan} follows from \eqref{eq:dirSumIso}.
\end{proof}

\begin{rem} \label{rem:dimRan}
\sloppy
$K$-invariant spaces are determined up to unitary equivalence by the dimensions of the spaces chosen by their range functions, in the following sense. Let $J_1$ and $J_2$ be two range functions in $\{L^2(K\backslash G; \H_\pi)\}_{\pi \in \hat{K}}$, and let $V_1$ and $V_2$ be the corresponding $K$-invariant subspaces of $L^2(G)$. Then there is a unitary map $U\colon V_1 \to V_2$ with the property that
\[ U L_\xi = L_\xi U \qquad (\xi \in K) \]
if and only if
\[ \dim J_1(\pi) = \dim J_2(\pi) \qquad (\pi \in \hat{K}). \]
This is a consequence of \eqref{eq:multRan}, since representations of compact groups are determined up to unitary equivalence by multiplicities of irreducible representations. Compare with Bownik's results on the dimension function for shift-invariant subspaces of $L^2(\R^n)$ \cite[Theorem 4.10]{B}.
\end{rem}

\begin{theorem} \label{thm:ranGen}
Let $\A \subset L^2(G)$ be an arbitrary family of functions, and let $S(\A) \subset L^2(G)$ be the $K$-invariant subspace generated by $\A$. That is,
\[ S(\A) = \overline{\spn}\{ L_\xi f : \xi \in K,\ f \in \A\}. \]
Then $S(\A) = V_J$, where
\[ J(\pi) = \overline{\spn}\{ \ran (Zf)(\pi) : f \in \A\} \qquad (\pi \in \hat{K}). \]
\end{theorem}

\begin{proof}
If $J$ and $J'$ are two range functions in $\{L^2(K\backslash G; \H_\pi)\}_{\pi \in \hat{K}}$ with the property that $J(\pi) \subset J'(\pi)$ for all $\pi \in \hat{K}$, then it is easy to see that $V_J \subset V_{J'}$. Moreover, $V_{J'}$ contains $S(\A)$ if and only if $J'(\pi)$ contains $\ran (Zf)(\pi)$ for all $f \in \A$, for every $\pi \in \hat{K}$. Since $S(\A)$ is the smallest $K$-invariant space containing $\A$, the corresponding range function $J$ must be such that $J(\pi)$ is the smallest closed subspace of  $L^2(K\backslash G; \H_\pi)$ containing $\ran (Zf)(\pi)$ for all $f \in \A$, for every $\pi \in \hat{K}$. That subspace is precisely
\[ \overline{\spn}\{ \ran (Zf)(\pi) : f \in \A\}. \qedhere\]
\end{proof}

\smallskip

\begin{cor} \label{cor:noCyc}
$L^2(G)$ contains a function $f$ with $\overline{\spn}\{L_\xi f : \xi \in K\} = L^2(G)$ if and only if $G=K$.
\end{cor}

\begin{proof}
The $K$-invariant space $L^2(G)$ corresponds with the range function $J'$ given by
\[ J'(\pi) = L^2(K\backslash G; \H_\pi) \qquad (\pi \in \hat{K}). \]
If $K \subsetneq G$, then any $f \in L^2(G)$ has
\[ \rank (Zf)(\pi) \leq d_\pi < \dim L^2(K\backslash G ; \H_\pi) \qquad (\pi \in \hat{K}). \]
By the previous theorem, the range function $J$ associated with $S(\{f\})$ has $J(\pi) = \ran (Zf)(\pi) \neq J'(\pi)$ for each $\pi \in \hat{K}$. Hence,
\[ S(\{f\}) = V_J \neq V_{J'} = L^2(G). \]

When $G=K$, on the other hand, it is well known that every subrepresentation of the regular representation is cyclic. See, for instance, \cite{GM}.
\end{proof}

We will now study the correspondence between range functions and $K$-invariant spaces in greater detail. Roughly speaking, we will see that the map $V_J \mapsto J$ allows us to view the lattice of $K$-invariant spaces as a much simpler lattice of linear subspaces. Many of the ideas that follow will appear again in our analysis of invariant subspaces of general representations of compact groups in Section \ref{sec:multGen}.

To begin, we introduce the notion of direct sum for range functions. If $J$ and $J'$ are two range functions in the same family $\mathscr{H} = \{\H(x)\}_{x\in X}$, with the property that $J(x) \perp J'(x)$ for every $x \in X$, then we say that $J$ and $J'$ are \emph{orthogonal}, and write $J \perp J'$. Given a family $\{J_\alpha\}_{\alpha \in A}$ of pairwise orthogonal range functions in $\mathscr{H}$, we denote $\oplus_{\alpha \in A} J_\alpha$ for the range function in $\mathscr{H}$ given by
\[ [\bigoplus_{\alpha \in A} J_\alpha](x) = \bigoplus_{\alpha \in A} [J_\alpha(x)] \qquad (x \in X). \]

Let $J$ and $J'$ be two range functions in $\{ L^2(K\backslash G; \H_\pi) \}_{\pi \in \hat{K}}$. For each $\pi \in \hat{K}$, we view $B(\H_\pi, J(\pi) )$ and $B(\H_\pi, J'(\pi))$ as closed subspaces of $B(\H_\pi, L^2(K\backslash G; \H_\pi))$, with the inner product
\[ \langle A,B \rangle =  d_\pi \langle A, B \rangle_{\HS}. \]
Then $B(\H_\pi, J(\pi) )$ is orthogonal to $B(\H_\pi, J'(\pi))$ if and only if $J(\pi) \perp J'(\pi)$. Since $Z$ is unitary and
\[ Z(V_J) = \bigoplus_{\pi \in \hat{K}} B(\H_\pi, J(\pi) ), \]
we conclude that
\begin{equation} \label{eq:ranPerp}
J \perp J' \iff V_J \perp V_{J'}.
\end{equation}
Moreover, if $\{J_\alpha\}_{\alpha \in A}$ is a family of range functions in $\{ L^2(K\backslash G; \H_\pi) \}_{\pi \in \hat{K}}$, then
\begin{equation} \label{eq:ranSum}
J = \bigoplus_{\alpha \in A} J_\alpha \iff V_J = \bigoplus_{\alpha \in A} V_{J_\alpha}.
\end{equation}

With these simple observations, we can easily describe all possible decompositions of $V_J$ as a direct sum of irreducible subspaces.

\begin{theorem} \label{thm:ranDecomp}
Let $J$ be a range function in $\{ L^2(K\backslash G; \H_\pi)\}_{\pi \in \hat{K}}$. For each $\pi \in \hat{K}$, choose an orthonormal basis $\{F_i^\pi\}_{i \in I_\pi}$ for $J(\pi)$.\footnote{If $J(\pi) = \{0\}$, we take $I_\pi$ to be the empty set.} Then
\[ V_{\pi,i} := \{ f \in L^2(G) : \ran (Zf)(\pi) \subset \spn\{F_i^\pi\},\text{ and }(Zf)(\sigma) = 0 \text{ for }\sigma \neq \pi\text{ in }\hat{K}\} \]
is an irreducible $K$-invariant space for each $\pi \in \hat{K}$ and $i \in I_\pi$, and
\begin{equation} \label{eq:ranDecomp1}
V_J = \bigoplus_{\pi \in \hat{K}} \bigoplus_{i \in I_\pi} V_{\pi,i}.
\end{equation}
Moreover, every decomposition of $V_J$ as a direct sum of irreducible $K$-invariant spaces occurs in this way.
\end{theorem}

In terms of the Zak transform, the direct sum decomposition \eqref{eq:ranDecomp1} simply says that
\[ Z(V_J) = \bigoplus_{\pi \in \hat{K}} \bigoplus_{i \in I_\pi} B(\H_\pi,\spn\{F_i^\pi\}). \]
We can think of $\spn\{F_i^\pi\}$ as being a copy of $\C$, so that $B(\H_\pi,\spn\{F_i^\pi\})$ is like a copy of $\H_\pi^*$. It will therefore come as no surprise that the corresponding action of $K$ on $B(\H_\pi, \spn\{F_i^\pi\})$ is unitarily equivalent to $\overline{\pi}$.

\begin{proof}
For each $\pi \in \hat{K}$ and each $i = 1,\dotsc,d_\pi$, let $J_{\pi,i}$ be the range function given by
\[ J_{\pi,i}(\sigma) = \begin{cases}
\spn\{F_i^\pi\}, & \text{if }\sigma = \pi \\
\{0\}, & \text{if }\sigma \neq \pi
\end{cases}  \qquad (\sigma \in \hat{K}). \]
Then $V_{\pi,i} = V_{J_{\pi,i}}$, and the direct sum decomposition \eqref{eq:ranDecomp1} follows immediately from \eqref{eq:ranSum}. If $\rho_{\pi,i}$ is the action of $K$ on $V_{\pi,i}$ by left translation, then $\rho_{\pi,i} \cong \overline{\pi}$ by  \eqref{eq:multRan}. In particular, $V_{\pi,i}$ is irreducible.

Suppose
\begin{equation} \label{eq:ranDecomp2}
V_J = \bigoplus_{\alpha \in A} V_\alpha
\end{equation}
is another decomposition of $V_J$ into irreducible $K$-invariant spaces. Each $V_\alpha$ has the form $V_{J_\alpha}$ for some range function $J_\alpha$, and \eqref{eq:multRan} shows that $J_\alpha(\pi)$ is one dimensional for exactly one $\pi \in \hat{K}$, and trivial for all others. For that unique value of $\pi$, we choose a unit norm vector $G_\alpha \in J_\alpha(\pi)$.

Applying \eqref{eq:ranSum} again, we see that $J = \bigoplus_{\alpha \in A} J_\alpha$. In particular,
\[ J(\pi) = \bigoplus_{\substack{\alpha \in A, \\ J_\alpha(\pi) \neq \{0\}}} J_\alpha(\pi) = \bigoplus_{\substack{\alpha \in A, \\ J_\alpha(\pi) \neq \{0\}}} \spn\{G_\alpha\} \]
for each $\pi \in \hat{K}$. Hence $\{G_\alpha : \alpha \in A,\, J_\alpha(\pi) \neq \{0\} \}$ is an orthonormal basis for $J(\pi)$. Rearranging the decomposition \eqref{eq:ranDecomp2} as
\[ V_J = \bigoplus_{\pi \in \hat{K}} \bigoplus_{\substack{\alpha \in A, \\ J_\alpha(\pi) \neq \{0\}}} V_\alpha \]
shows it has the same form as \eqref{eq:ranDecomp1}.
\end{proof}

\medskip

\section{Frames of translates} \label{sec:tranFrm}

There is a long tradition of combining range function classifications of invariant spaces with conditions for a family of translates to form a reproducing system. Bownik \cite{B} seems to have the first results along these lines. His example was followed in \cite{BHP2,BR,CP,CMO,I,KR}. We now carry that tradition to the setting of compact, nonabelian subgroups. For our purposes, the relevant notion will be a continuous version of frames.

\begin{defn}
Let $\H$ be a separable Hilbert space, and let $(\M,\mu)$ be a $\sigma$-finite measure space. Let $\{f_x\}_{x\in \M}$ be an indexed family with the property that $x \mapsto \langle g, f_x\rangle$ is a measurable function on $\M$ for every $g \in \H$. Then $\{f_x\}_{x\in \M}$ is a \emph{Bessel mapping} if there is a constant $B > 0$ such that
\[ \int_\M | \langle g, f_x \rangle |^2\, d\mu(x) \leq B \Norm{g}^2 \qquad \text{for every }g \in \H. \]
It is a \emph{continuous frame} for $\H$ if there are constants $0 < A \leq B < \infty$ such that 
\[ A \Norm{g}^2 \leq \int_\M | \langle g, f_x \rangle |^2\, d\mu(x) \leq B \Norm{g}^2 \qquad \text{for every }g \in \H. \]
The constants $A$ and $B$ are called \emph{bounds}. If we can take $A=B$, the frame is \emph{tight}. If we can take $A=B=1$, it is a \emph{Parseval frame}.
\end{defn}

The reader unfamiliar with this notion may consult \cite{AAG,K}, where it was originally developed. Further details are available in \cite{GH} and \cite{RND}. In the case where $\M$ is a discrete set equipped with counting measure, continuous frames reduce to the usual, discrete version. (The reader may even take this as a definition.) We will use the terms ``frame'' and ``continuous frame'' interchangeably.

The usual reproducing properties of discrete frames carry over to the continuous versions, with predictable modifications. Let $\{f_x \}_{x\in \M}$ be a Bessel mapping. The associated \emph{analysis operator} $T \colon \H \to L^2(\M)$ is defined by
\[ (Tg)(x) = \langle g, f_x \rangle \qquad (g \in \H,\ x \in \M); \]
its adjoint is the \emph{synthesis operator} $T^*\colon L^2(\M) \to \H$, 
\[ T^* \phi = \int_\M \phi(x) f_x\, d\mu(x) \qquad (\phi \in L^2(\M)), \]
where the vector-valued integral is interpreted in the weak sense. The \emph{Gramian} is $\G = T T^*$, and the \emph{frame operator} is $S = T^* T$. When our Bessel mapping is a continuous frame, the frame operator is positive and invertible, and $\{S^{-1/2} f_x \}_{x\in \M}$ is a continuous Parseval frame for $\H$, called the \emph{canonical tight frame}. For Parseval frames, the frame operator is the identity map, and the Gramian is an orthogonal projection. Even when the frame is not tight, $\{S^{-1} f_x\}_{x\in \M}$ is another frame for $\H$ which satisfies
\[ g = \int_\M \langle g, S^{-1} f_x \rangle f_x\, d\mu(x) \qquad (g \in \H). \]

\begin{rem} \label{rem:frmBnds}
The results in this paper apply for arbitrary second countable compact groups, which includes finite groups in particular. When $K$ is finite, all of our results about continuous frames indexed by $K$ can be interpreted in terms of discrete frames. We caution that it is necessary to reinterpret the frame bounds in this case, since Haar measure on $K$ is normalized so that $|K|=1$. In the special case where $K$ is finite, a continuous frame over $K$ having bounds $A,B$ is the same as a discrete frame indexed by $K$ having bounds $\card(K)\cdot A, \card(K)\cdot B$.
\end{rem}

For a countable family $\A \subset L^2(G)$, we will denote
\[ E(\A) = \{L_\xi f\}_{\xi \in K, f \in \A} \]
for the translates of $\A$. Recall that
\[ S(\A) = \overline{\spn}\, E(\A) \]
is the $K$-invariant space generated by $\A$, and that $S(\A) = V_J$, with
\begin{equation} \label{eq:SARan}
J(\pi) = \overline{\spn} \{ \ran (Zf)(\pi) : f \in \A \} \qquad (\pi \in \hat{K}).
\end{equation}
We would like to know under what circumstances $E(\A)$ forms a continuous frame for $S(\A)$. Our main result is as follows.

\begin{theorem} \label{thm:tranFrm}
Let $\A \subset L^2(G)$ be a countable family of functions, and let $J$ be the range function in \eqref{eq:SARan}. For any constants $0 < A \leq B < \infty$ and any choice of orthonormal bases $e_1^\pi,\dotsc,e_{d_\pi}^\pi \in \H_\pi$, $\pi \in \hat{K}$, the following are equivalent.
\begin{enumerate}[(i)]
\item $E(\A)$ is a continuous frame for $S(\A)$ with bounds $A,B$. That is,
\begin{equation} \label{eq:tranFrm0}
A \Norm{g}^2 \leq \sum_{f \in \A} \int_K |\langle g, L_\xi f \rangle|^2\, d\xi \leq B \Norm{g}^2 \qquad (g \in S(\A)).
\end{equation}
\item For every $\pi \in \hat{K}$, $\{(Zf)(\pi) e_i^\pi : f \in \A, i =1,\dotsc,d_\pi\}$ is a discrete frame for $J(\pi)$ with bounds $A,B$.
\end{enumerate}
\end{theorem}

This is in the spirit of \cite[Theorem 2.3]{B}. When $K$ is compact and \emph{abelian}, the theorem above reduces to \cite[Theorem 5.4]{I}. If $G$ is also abelian, the same result was given in \cite[Theorem 6.10]{BHP2}. Similar results appear in \cite{BHP2,BR,CP,CMO,I,JaLem,KR,RS}.

The proof of Theorem \ref{thm:tranFrm} relies on the following lemma, which will also play a prominent role in Section \ref{sec:brack}. To each pair $f,g \in L^2(G)$, we associate the \emph{matrix element} $V_f g \in C(K)$ given by
\[ (V_f g)(\xi) = \langle g, L_\xi f \rangle \qquad (\xi \in K). \]

\begin{lemma} \label{lem:ZakBrack}
For $f, g \in L^2(G)$ and $\pi \in \hat{K}$,
\[ (V_f g)\caret(\pi) = (Zf)(\pi)^* (Zg)(\pi). \]
\end{lemma}

\begin{proof}
Fix an orthonormal basis $e_1^\pi,\dotsc,e_{d_\pi}^\pi$ for each $\H_\pi$, $\pi \in \hat{K}$. For $f, g \in L^2(G)$, $\pi \in \hat{K}$, and $i,j=1,\dotsc,d_\pi$, the $(i,j)$-entry of the matrix for $(V_f g)\caret(\pi)$ with respect to this basis is
\[ (V_f g)\caret(\pi)_{i,j} = \int_K \int_G g(x) \overline{(L_\xi f)(x)}\, dx\, \overline{\pi_{i,j}(\xi)}\, d\xi. \]
Applying the measure space isomorphism $G \to K\backslash G \times K$ from \eqref{eq:measIsom}, we see this is equal to
\[ \int_K \int_{K\backslash G} \int_K g_{Kx}(\eta) \overline{ f_{Kx}(\xi^{-1} \eta) }\, d\eta\, d(Kx)\, \overline{\pi_{i,j}(\xi)}\, d\xi = \int_K \int_{K\backslash G} (g_{Kx} * f_{Kx}^*)(\xi)\, d(Kx)\, \overline{\pi_{i,j}(\xi)}\, d\xi,  \]
where $f_{Kx}$ and $g_{Kx}$ are as defined in \eqref{eq:cosetFunct}. We wish to reverse the order of integration above with Fubini's Theorem. Assuming for the moment that this is possible, we will have
\[ (V_f g)\caret(\pi)_{i,j} = \int_K \int_{K\backslash G} (g_{Kx} * f_{Kx}^*)(\xi)\, d(Kx)\, \overline{\pi_{i,j}(\xi)}\, d\xi \]
\[ = \int_{K\backslash G} \int_K  (g_{Kx} * f_{Kx}^*)(\xi) \overline{\pi_{i,j}(\xi)}\, d\xi\, d(Kx)  = \int_{K\backslash G} (g_{Kx} * f_{Kx}^*)\caret(\pi)_{i,j}\, d(Kx) \]
\[ = \int_{K\backslash G} \langle (f_{Kx})\caret(\pi)^* (g_{Kx})\caret(\pi) e_j^\pi, e_i^\pi \rangle\, d(Kx) = \int_{K\backslash G} \langle (g_{Kx})\caret(\pi) e_j^\pi, (f_{Kx})\caret(\pi) e_i^\pi \rangle\, d(Kx) \]
\[  = \int_{K\backslash G} \langle [(Zg)(\pi)e_j^\pi](Kx), [(Zf)(\pi) e_i^\pi](Kx) \rangle\, d(Kx) = \langle (Zg)(\pi) e_j^\pi, (Zf)(\pi) e_i^\pi \rangle \]
\[  = [(Zf)(\pi)^* (Zg)(\pi)]_{i,j}, \]
where we have applied the definition of the Zak transform \eqref{eq:ZakFour} in the third to last equality. Once the above holds for all $i$ and $j$, we will be able to conclude that
\[ (V_f g)\caret(\pi) = (Zf)(\pi)^* (Zg)(\pi), \]
as desired.

It only remains to justify our use of Fubini's Theorem. To do so, we observe first that
\[ |\pi_{i,j}(\xi)| = | \langle \pi(\xi) e_j^\pi, e_i^\pi \rangle | \leq \Norm{\pi(\xi) e_j^\pi} \Norm{e_i^\pi} = 1 \qquad (\xi \in K), \]
by Cauchy-Schwarz. Hence,
\[ \int_{K\backslash G} \int_K | (g_{Kx} * f_{Kx}^*)(\xi) \pi_{i,j}(\xi)|\, d\xi\, d(Kx) \leq \int_{K\backslash G} \Norm{g_{Kx} * f_{Kx}^*}_{L^1(K)}\, d(Kx) \]
\[ \leq \int_{K\backslash G} \Norm{g_{Kx}}_{L^1(K)} \Norm{f_{Kx}}_{L^1(K)}\, d(Kx) \leq \left( \int_{K\backslash G} \Norm{g_{Kx}}_{L^1(K)}^2\, d(Kx) \right)^{1/2} \left( \int_{K\backslash G} \Norm{f_{Kx}}_{L^1(K)}^2\, d(Kx) \right)^{1/2}. \]

The proof will be finished if we can show that $\int_{K\backslash G} \Norm{f_{Kx}}_{L^1(K)}^2 d(Kx) < \infty$ for all $f \in L^2(G)$. An application of Minkowski's Integral Inequality produces
\[ \left( \int_{K\backslash G} \Norm{ f_{Kx} }_{L^1(K)}^2\, d(Kx) \right)^{1/2} = \left( \int_{K\backslash G} \left| \int_K |f(\eta \tau(Kx))|\, d\eta \right|^2 d(Kx) \right)^{1/2} \]
\[ \leq \int_K \left( \int_{K\backslash G} | f(\eta \tau(Kx)) |^2\, d(Kx) \right)^{1/2} d\eta. \]
Let
\begin{align*}
E &= \{ \eta \in K : \int_{K\backslash G} | f(\eta \tau(Kx) |^2\, d(Kx) < 1\} \\
F &= \{ \eta \in K : \int_{K\backslash G} |f(\eta \tau(Kx) |^2\, d(Kx) \geq 1\}.
\end{align*}
(These are well defined up to sets of measure zero.) Then
\[  \left( \int_{K\backslash G} \Norm{ f_{Kx} }_{L^1(K)}^2\, d(Kx) \right)^{1/2} \leq \int_K \left( \int_{K\backslash G} | f(\eta \tau(Kx)) |^2\, d(Kx) \right)^{1/2} d\eta \]
\[ = \int_E \left( \int_{K\backslash G} | f(\eta \tau(Kx)) |^2\, d(Kx) \right)^{1/2} d\eta + \int_F \left( \int_{K\backslash G} | f(\eta \tau(Kx)) |^2\, d(Kx) \right)^{1/2} d\eta \]
\[ \leq |E| + \int_F  \int_{K\backslash G} | f(\eta \tau(Kx)) |^2\, d(Kx) \, d\eta \leq 1 + \int_K \int_{K\backslash G} | f(\eta \tau(Kx)) |^2\, d(Kx) \, d\eta \]
\[ = 1 + \int_G | f(x) |^2\, dx < \infty, \]
where we have once again applied the measure space isomorphism $K \times K\backslash G \to G$. This completes the proof.
\end{proof}

With this lemma in hand, Theorem \ref{thm:tranFrm} becomes an easy consequence of Plancherel's Theorem and our classification of $K$-invariant spaces.

\medskip

\begin{proof}[Proof of Theorem \ref{thm:tranFrm}] For any $f,g \in L^2(G)$, we use Plancherel's Theorem and Lemma \ref{lem:ZakBrack}  to perform the fundamental calculation
\begin{equation}\label{eq:tranFrm1}
\int_K |\langle g, L_\xi f \rangle |^2\, d\xi =  \sum_{\pi \in \hat{K}} d_\pi \Norm{ (Zf)(\pi)^* (Zg)(\pi)}_{\HS}^2 = \sum_{\pi \in \hat{K}} d_\pi \sum_{j=1}^{d_\pi} \sum_{i=1}^{d_\pi} |\langle (Zg)(\pi) e_j^\pi, (Zf)(\pi) e_i^\pi \rangle |^2.
\end{equation}
On the other hand, the fact that $Z$ is unitary implies
\begin{equation} \label{eq:tranFrm2}
\Norm{g}^2 = \sum_{\pi \in \hat{K}} d_\pi \Norm{ (Zg)(\pi) }_{\HS}^2 = \sum_{\pi \in \hat{K}} d_\pi \sum_{j=1}^{d_\pi} \Norm{ (Zg)(\pi)e_j^\pi }^2.
\end{equation}

Suppose (i) holds. Fix $\pi \in \hat{K}$, and choose any $G \in J(\pi)$. Define $g \in L^2(G)$ by the formula
\[ (Zg)(\sigma)e_j^\sigma = \begin{cases} 
d_\pi^{-1/2} G, & \text{if }\sigma = \pi \\ 
0, & \text{if }\sigma \neq \pi
\end{cases} \qquad (\sigma \in \hat{K};\ j = 1,\dotsc,d_\sigma). \]
Then $g \in V_J = S(\A)$, by construction. It satisfies
\[ \Norm{g}^2 =\Norm{G}^2, \]
by \eqref{eq:tranFrm2}, and
\[ \sum_{f\in \A} \int_K | \langle g, L_\xi f \rangle |^2\, d\xi = \sum_{f\in \A}  \sum_{i=1}^{d_\pi} |\langle G, (Zf)(\pi) e_i^\pi \rangle |^2, \]
by \eqref{eq:tranFrm1}. Substituting these equations into \eqref{eq:tranFrm0} gives
\[ A \Norm{G}^2 \leq \sum_{f\in \A}  \sum_{i=1}^{d_\pi} |\langle G, (Zf)(\pi) e_i^\pi \rangle |^2 \leq B \Norm{G}^2. \]
In other words, (ii) holds.

Now assume (ii) is satisfied. For every $g \in S(\A) = V_J$ and every $\pi \in \hat{K}$, $(Zg)(\pi)e_j^\pi \in J(\pi)$. By \eqref{eq:tranFrm2} and the frame inequality,
\[ A \Norm{g}^2 = \sum_{\pi \in \hat{K}} d_\pi \sum_{j=1}^{d_\pi} A \Norm{ (Zg)(\pi)e_j^\pi }^2 \leq \sum_{\pi \in \hat{K}} d_\pi \sum_{j=1}^{d_\pi} \sum_{f \in \A} \sum_{i=1}^{d_\pi} |\langle (Zg)(\pi) e_j^\pi, (Zf)(\pi) e_i^\pi \rangle |^2. \]
Applying \eqref{eq:tranFrm1} to the last expression above, we see that
\[ A \Norm{g}^2 \leq \sum_{f \in \A} \int_K |\langle g, L_\xi f \rangle |^2\, d\xi. \]
A similar computation produces
\[ \sum_{f \in \A} \int_K |\langle g, L_\xi f \rangle |^2\, d\xi \leq B \Norm{g}^2. \]
This proves (i).
\end{proof}

\medskip

\section{Bracket analysis for compact group actions} \label{sec:brack}

We turn our attention now to a detailed study of group frames, as described in the introduction. In this section, we introduce a computational system known as a \emph{bracket} for the analysis of representations of compact groups. Our primary motivation is the study of group frames with a single generator. We will see, however, that the bracket carries vital information about the structure of the representation itself, including its isotypical components and the multiplicities of irreducible representations. Several applications for the theory of group frames, including a complete classification of (compact) group frames with a single generator, appear in Section \ref{sec:brackApp}. Throughout, we fix a second countable compact group $K$, as in the previous sections, with Haar measure normalized so that $|K|=1$. We also fix a unitary representation $\rho$ of $K$, acting on a separable Hilbert  space $\H_\rho$.

Our approach is motivated by the work of Weiss, et al.\ in \cite{HSWW2}. Let $\G$ be a second countable locally compact abelian (LCA) group, with dual group $\hat{\G}$. Normalize Haar measures on $\G$ and $\hat{\G}$ so that the Plancherel theorem holds. A representation $\pi\colon \G \to U(\H_\pi)$ is called \emph{dual integrable} if there is a \emph{bracket}
\[ [ \cdot, \cdot] \colon \H_\pi \times \H_\pi \to L^1(\hat{\G}) \]
such that
\[ \langle f, \pi(x) g \rangle = \int_{\hat{\G}} [f, g](\alpha) \overline{\alpha(x)}\, d\alpha \qquad (f, g \in \H_\pi;\ x \in \G). \]
When $\G$ is identified with the dual of $\hat{\G}$ via Pontryagin Duality, this means that $\langle f, \pi(\cdot) g \rangle$ is the Fourier transform of $[f,g]$. The bracket provides an elegant description of frame properties for an orbit $\{\pi(x) f\}_{x\in \G}$.

\begin{prop}[ \cite{HSWW2,I} ] \label{prop:brackFrm}
For $f \in \H_\pi$ and constants $A,B$ with $0 < A \leq B < \infty$, the following are equivalent.
\begin{enumerate}[(i)]
\item The orbit $\{\pi(x) f\}_{x\in \G}$ is a continuous frame for its closed linear span, with bounds $A,B$
\item For a.e.\ $\alpha \in \hat{\G}$, either $[f,f](\alpha) = 0$ or $A \leq [f,f](\alpha) \leq B$.
\end{enumerate}
\end{prop}

A possible difficulty with this approach is that, generally speaking, one may know a representation is dual integrable without being able to compute the bracket.\footnote{For certain kinds of representations, there are ways to recover the bracket even when $\G$ is not compact. Most of these methods involve variants of the Zak transform. See \cite{HSWW2} and \cite{I}.} Suppose, however, that $\G$ is \emph{compact} abelian. Then we can compute brackets as follows. Let $\pi$ be \emph{any} unitary representation of $\G$ on a separable Hilbert space $\H_\pi$. Then $\pi$ decomposes as a direct sum of cyclic subrepresentations, each of which is unitarily equivalent to a subrepresentation of the regular representation. (See, for instance, \cite{GM}.) By \cite[Corollary 3.4]{HSWW2}, $\pi$ is dual integrable. Let $[\cdot,\cdot] \colon \H_\pi \times \H_\pi \to L^1(\hat{\G})$ be a bracket for $\pi$. That is,
\[ \langle f , \pi(x) g \rangle = [f,g]\caret(x) \qquad (f,g \in \H;\ x \in \G). \]
Since $\G$ is compact, $[f,g]\caret$ lies in $C(\G) \subset L^1(\G)$ for every $f,g \in \H_\pi$. Therefore we can apply Fourier inversion to recover the bracket from the matrix elements $\langle f, \pi(\cdot) g \rangle$:
\[ [f,g](\alpha) = \langle f, \pi(\cdot) g \rangle\caret(\alpha^{-1}) \qquad (f,g \in \H_\pi;\ \alpha \in \hat{\G}). \]

These results suggest that, for our general compact group $K$ with unitary representation $\rho$, it should be possible to analyze frames appearing as orbits of $\rho$ using the (operator-valued) Fourier transform of the matrix elements 
\[ (V_g f)(\xi) := \langle f, \rho(\xi) g \rangle \qquad (f,g \in \H_\rho;\ \xi \in K). \]
This is indeed the case.

\begin{defn}
The \emph{bracket} associated with $\rho$ is the map
\[ [ \cdot, \cdot ] \colon \H_\rho \times \H_\rho \to \bigoplus_{\pi \in \hat{K}} B(\H_\pi) \]
given by
\[ [f,g](\pi) = (V_g f)\caret(\pi) \qquad (\pi \in \hat{K}). \]
\end{defn}

Here, as elsewhere, we consider $B(\H_\pi)$ to be a Hilbert space with inner product given by
\[ \langle A, B \rangle = d_\pi \langle A, B \rangle_{\HS} = d_\pi \tr(B^* A). \]
Then $\bigoplus_{\pi \in \hat{K}} B(\H_\pi)$ is the Hilbert space direct sum.

Following the notation of \cite{HSWW2}, we will denote $\langle f \rangle \subset \H_\rho$ for the cyclic subspace generated by $f \in \H_\rho$. That is,
\[ \langle f \rangle = \overline{\spn} \{ \rho(\xi) f : \xi \in K\} \qquad (f \in \H_\rho). \]
Our main result is the following.

\begin{theorem} \label{thm:brackFrm}
For $f \in \H_\rho$ and constants $A,B$ with $0 < A \leq B < \infty$, the following are equivalent.
\begin{enumerate}[(i)]
\item The orbit $\{ \rho(\xi) f \}_{\xi\in K}$ is a continuous frame for $\langle f \rangle$ with bounds $A, B$.
\item For every $\pi \in \hat{K}$, the nonzero eigenvalues of $[f,f](\pi)$ lie in the interval $[A,B]$.
\end{enumerate}
\end{theorem}

When $\dim \H_\rho < \infty$, it is easy to tell when $\langle f \rangle = \H_\rho$ using the ranks of $[f,f](\pi)$, $\pi \in \hat{K}$; see Proposition \ref{prop:cycBrack} below. Thus, one can tell whether or not $\{\rho(\xi) f \}_{\xi \in K}$ is a frame for $\H_\rho$, and with what bounds, based solely on the eigenvalues of $[f,f](\pi)$, $\pi \in \hat{K}$, and their multiplicities. The condition that $\dim \H_\rho < \infty$ is always satisfied when $\{\rho(\xi) f \}_{\xi \in K}$ is a frame for $\H_\rho$; this is a consequence of Theorem \ref{thm:KFrmDim}, infra.

If $Q_\pi$ denotes orthogonal projection of $\H_\pi$ onto $(\ker [f,f](\pi))^\perp$, then condition (ii) of the theorem above can be interpreted to say that $A Q_\pi \leq [f,f](\pi) \leq B Q_\pi$ for each $\pi \in \hat{K}$. (Compare with \cite[Theorem A]{BHP}.) In the special case where $K$ is compact \emph{abelian}, Theorem \ref{thm:brackFrm} reduces to Proposition \ref{prop:brackFrm}.

Tight frames generated by actions of \emph{finite} nonabelian groups have been the focus of a flurry of recent activity \cite{CW,CW2,Ha,VW2,VW,VW3}. See \cite[Theorem 6.18]{VW2} and its generalization \cite[Theorem 2.8]{VW3} in particular for another characterization of \emph{tight} frames that occur in this way. A nice summary of the state of the art circa 2013 appears in \cite{Wa}; unfortunately the survey is already out of date, thanks in part to recent work by Waldron himself. This field is advancing rapidly.

Brackets have been used to analyze reproducing systems in $L^2(\R^n)$ since at least the work of Jia and Micchelli \cite{JM}. Weiss and his collaborators brought these techniques into the group-theoretic domain with \cite{HSWW2}, as described above. In the nonabelian setting, Hern{\'a}ndez, et al.\ have developed notions of bracket maps for the Heisenberg group and for countable discrete groups \cite{BHM,BHP,BHP3}.

The bracket defined above is related to the one that appears in \cite{BHP,BHP3}. Suppose that $K$ is finite (that is, both compact and discrete). Let us write $[\cdot,\cdot]_0 \colon \H_\rho \times \H_\rho \to B(L^2(K))$ for the bracket as developed in \cite{BHP}. One can show that, for all $f,g \in \H_\rho$,
\[ [f,g]_0(\phi) = \phi * V_g f \qquad (\phi \in L^2(K)). \]
Conjugating with the Fourier transform turns $[f,g]_0$ into left multiplication by $[f,g]$. 
One might say the papers \cite{BHP,BHP3} study the convolution operator given by $V_g f$, where this paper studies its Fourier transform. 

Much of our analysis relies on functions of positive type. We remind the reader that $\phi \in C(K)$ is said to be of \emph{positive type} if
\[ \int_K (f * f^*)(\xi)\, \phi(\xi) \, d\xi \geq 0 \qquad \text{for all }f\in L^1(K). \]
Equivalently, there is a unitary representation $\sigma$ of $K$ and a vector $f \in \H_\sigma$ such that
\[ \phi(\xi) = \langle f, \sigma(\xi) f \rangle \qquad (\xi \in K). \]
The representation and the vector are unique in the following sense: If $\sigma'$ is another representation of $K$ with a cyclic vector $f' \in \H_{\sigma'}$ such that $\phi(\xi) = \langle f', \sigma'(\xi) f' \rangle$ for all $\xi \in K$, then there is a unitary $U \colon \H_{\sigma'} \to \H_\sigma$ intertwining $\sigma'$ with $\sigma$ and mapping $f' \mapsto f$. (See, for instance, \cite[\S 3.3]{F}.) When $\sigma$ is the regular representation and $f, g \in L^2(K)$, we have
\begin{equation} \label{eq:transPos}
\langle f, L_\xi g \rangle = \int_K f(\eta) g^*(\eta^{-1} \xi)\, d\eta = (f * g^* )(\xi) \qquad (\xi \in K).
\end{equation}
For arbitrary $f \in L^2(K)$, this means that $\phi = f * f^*$ is a function of positive type. Up to unitary equivalence, the cyclic representations of $K$ are precisely the subrepresentations of the regular representation (\cite{GM}); thus \emph{every} function of positive type takes this form. In particular, 
\begin{equation} \label{eq:posSelfAdj}
\phi^* = \phi,
\end{equation}
and
\begin{equation}\label{eq:posFour}
\hat{\phi}(\pi) = (f * f^*)\caret (\pi) = \hat{f}(\pi)^* \hat{f}(\pi) \geq 0 \qquad (\pi \in \hat{K}).
\end{equation}
(It is positive semidefinite.)

The bracket $[f,f]$ in Theorem \ref{thm:brackFrm} is the Fourier transform of the associated function of positive type 
\[ V_f f(\xi) = \langle f, \rho(\xi) f \rangle \qquad (\xi \in K). \]
Given $V_f f$, it is possible to reconstruct the Hilbert space $\langle f \rangle$, the restriction of $\rho$ to $\langle f \rangle$, and the cyclic vector $f$. In other words, $V_f f$ contains complete information about the cyclic representation generated by $f$. Philosophically speaking, it must also be able to tell us when the orbit of $f$ is a continuous frame for $\langle f \rangle$. Theorem \ref{thm:brackFrm} tells how to extract this information.

\medskip

We will write $A^\dagger$ for the Moore-Penrose pseudoinverse of a bounded linear operator $A$. When $A$ has closed range, $AA^\dagger$ is orthogonal projection onto the range of $A$, and $A^\dagger A$ is orthogonal projection onto $(\ker A)^\perp$.

\begin{lemma} \label{lem:cycIsom}
For every $f \in \H_\rho$, there is a unique linear isometry $T_f \colon \langle f \rangle \to L^2(K)$ intertwining $\rho$ with left translation, and sending $f$ to a function of positive type.
Explicitly,
\begin{equation}\label{eq:cycIsom1}
(T_fg)\caret(\pi) = ( [f,f](\pi)^{1/2} )^\dagger\cdot [g,f](\pi) \qquad (g \in \langle f \rangle,\ \pi \in \hat{K}).
\end{equation}
\end{lemma}

\begin{proof}
Since the restriction of $\rho$ to $\langle f \rangle$ is square integrable, the existence of a linear isometry $T_f \colon \langle f \rangle \to L^2(K)$ intertwining $\rho$ with left translation and mapping $f$ to a function of positive type is given by \cite[Theorem 13.8.6]{D}. Then $(T_f f)^* = T_f f$, and
\[ (V_f f)(\xi) = \langle T_f f, L_\xi (T_f f) \rangle = [T_f f * (T_f f)^*](\xi) = (T_f f * T_f f)(\xi) \qquad (\xi \in K). \]
Since $(T_f f)\caret(\pi) \geq 0$ for all $\pi \in \hat{K}$, we conclude that
\[ (T_f f)\caret(\pi) = [f,f](\pi)^{1/2} \qquad (\pi \in \hat{K}). \]
For any $g \in \langle f \rangle$, \eqref{eq:transPos} gives
\[ (V_f g)(\xi) = \langle T_f g, L_\xi T_f f \rangle = [(T_f g) * (T_f f)^*](\xi) = [(T_f g) * (T_f f)](\xi) \qquad (\xi \in K), \]
or equivalently,
\begin{equation} \label{eq:cycIsom2}
[g,f](\pi) = (T_f f)\caret(\pi)\cdot (T_f g)\caret(\pi) \qquad (\pi \in \hat{K}).
\end{equation}
Since $T_f g \in \langle T_f f \rangle$, Theorem \ref{thm:ranGen} shows that
\[ \ran (T_f g)\caret(\pi) \subset \ran (T_f f)\caret(\pi) = (\ker (T_f f)\caret(\pi) )^\perp \qquad (\pi \in \hat{K}). \]
(Here we use the Fourier transform in place of the Zak transform; see Remark \ref{rem:ZakFour}.) Applying $[(T_f f)\caret(\pi)]^\dagger = ([f,f](\pi)^{1/2})^\dagger$ to both sides of \eqref{eq:cycIsom2} establishes \eqref{eq:cycIsom1}. In particular, $T_f$ is uniquely determined.
\end{proof}

\smallskip

\begin{prop} \label{prop:brackProp1}
The bracket has the following properties.
\begin{enumerate}[(i)]
\item $[\cdot,\cdot]$ is linear in the first variable, and conjugate linear in the second.

\item For all $f, g \in \H_\rho$ and $\pi \in \hat{K}$,
\[ [f,g](\pi) = [g,f](\pi)^*. \]

\item For all $f \in \H_\rho$ and $\pi \in \hat{K}$, $[f,f](\pi) \geq 0$.

\item For all $f,g \in \H_\rho$ and $A \in B(\H_\rho)$, 
\[ [Af, g ] = [f,A^* g]. \]

\item For all $f,g \in \H_\rho$, $\pi \in \hat{K}$, and $\xi \in K$,
\[ [f, \rho(\xi) g](\pi) = \pi(\xi)\cdot [f,g](\pi) \]
and
\[ [\rho(\xi) f, g](\pi) = [f,g](\pi)\cdot \pi(\xi^{-1}). \]

\item For $f,g \in \H_\rho$, $f \perp \langle g \rangle$ if and only if $[f,g]=0$.

\end{enumerate}
\end{prop}

More properties will be given in Propositions \ref{prop:brackProp2} and \ref{prop:brackProp3} below.

\begin{proof}
Item (i) follows from linearity of the Fourier transform and sesquilinearity of the map $(f,g) \mapsto V_g f$. To see (ii), apply  \eqref{eq:FourConv} to the identity $V_f g = (V_g f)^*$. Equation \eqref{eq:posFour} gives (iii), since $V_f f$ is a function of positive type. Apply the simple identity $V_g (Af) = V_{A^* g}f$ to get (iv). For (v), use \eqref{eq:FourTrans} and the identities 
\[ V_{\rho(\xi) g} f = R_\xi (V_g f), \quad V_g (\rho(\xi) f) = L_\xi (V_g f) \qquad (f,g \in \H_\rho;\ \xi \in K). \]

For (vi), first assume that $f \perp \langle g \rangle$. Let $P_g$ denote orthogonal projection of $\H_\rho$ onto $\langle g \rangle$, and apply (iv) to see that
\[ [f, g] = [f, P_g g] = [P_g f, g] = 0. \]
Now suppose that $f,g \in \H_\rho$ satisfy $[f,g]=0$. By Plancherel's Theorem, $V_g f =0$. That is, $\langle f, \rho(\xi) g \rangle = 0$ for all $\xi \in K$. Hence $f \perp \langle g \rangle$.
\end{proof}

When $K$ is contained in a larger second countable, locally compact group $G$, the Zak transform provides a bracket for the action of $K$ on $L^2(G)$ by left translation. Indeed, Lemma \ref{lem:ZakBrack} says precisely that
\[ [f,g](\pi) = (Zg)(\pi)^* (Zf)(\pi) \qquad (f,g \in L^2(G);\ \pi \in \hat{K}) \]
in this case. The theorem below shows that this example is universal; it is always possible to embed $\H_\rho$ as a $K$-invariant subspace of $L^2(G)$, for some larger group $G$ containing $K$, in such a way that $\rho$ becomes left translation by $K$.

\begin{theorem} \label{thm:isom}
There is a second countable, locally compact group $G$ containing $K$ as a closed subgroup, and a linear isometry $T \colon \H_\rho \to L^2(G)$ satisfying
\[ T \rho(\xi) f = L_\xi T f \qquad (f \in \H_\rho,\ \xi \in K). \]
If $Z$ is the Zak transform for the pair $(G,K)$, then the bracket for $\rho$ is given by
\[ [f,g](\pi) = (ZTg)(\pi)^*(ZTf)(\pi) \qquad (f,g \in \H_\rho;\ \pi \in \hat{K}). \]
\end{theorem}

\begin{proof}
There is a countable family $\{f_i\}_{i \in I} \subset \H_\rho$ for which
\[ \H_\rho = \bigoplus_{i \in I} \langle f_i \rangle. \]
For each $i \in I$, let $T_{f_i} \colon \langle f_i \rangle \to L^2(K)$ be the isometry from Lemma \ref{lem:cycIsom}. Give $I$ the structure of a discrete abelian group, and let $G = K \times I$. Given $g \in \H_\rho$, find the unique decomposition $g = \sum_{i \in I} g_i$ with $g_i \in \langle f_i \rangle$ for all $i$, and define
\[ (Tg)(\xi, i) = (T_{f_i} g_i)(\xi) \qquad (\xi \in K,\ i \in I). \]
Then $T\colon \H_\rho \to L^2(G)$ is the desired isometry.
\end{proof}

\begin{prop} \label{prop:brackProp2}
In addition to the properties listed in Proposition \ref{prop:brackProp1}, the bracket satisfies the following.
\begin{enumerate}[(i)]
\item For all $f,g \in \H_\rho$,
\begin{equation} \label{eq:innBrack}
\langle f, g \rangle = \sum_{\pi \in \hat{K}} d_\pi \tr( [f,g](\pi) ).
\end{equation}

\item For all $f, g \in \H_\rho$,
\begin{equation} \label{eq:brackProp2.1}
\Norm{ [f,g](\pi) }_{\HS}^2 \leq \Norm{ [f,f](\pi) }_{\HS} \Norm{ [g,g](\pi) }_{\HS} \qquad (\pi \in \hat{K}).
\end{equation}

\item If $f_n \to f$ in $\H_\rho$, then $[f_n,g] \to [f,g]$ for all $g \in \H_\rho$. In particular, 
\[ [f_n,g](\pi) \to [f,g](\pi) \]
for all $g \in \H_\rho$ and $\pi \in \hat{K}$.

\end{enumerate}
\end{prop}

\begin{proof}
By applying Theorem \ref{thm:isom} if necessary, we may assume that $\H_\rho$ is a $K$-invariant subspace of $L^2(G)$ for some second countable locally compact group $G$ containing $K$ as a closed subgroup, that $\rho$ is given by left translation of $K$, and that
\[ [f,g](\pi) = (Zg)(\pi)^* (Zf)(\pi) \qquad (f,g \in \H_\rho;\ \pi \in \hat{K}), \]
where $Z$ is the Zak transform for the pair $(G,K)$. Now (iii) follows immediately from continuity of the Zak transform.

To prove (i), we simply compute
\[ \langle f, g \rangle = \langle Zf, Zg \rangle = \sum_{\pi \in \hat{K}} d_\pi \langle (Zf)(\pi), (Zg)(\pi) \rangle_{\HS} = \sum_{\pi \in \hat{K}} d_\pi \tr([f,g](\pi)) \qquad (f,g \in \H_\rho). \]

For (ii), we use Cauchy-Schwarz for the Hilbert-Schmidt inner product to estimate
\[ \Norm{ [f,g](\pi) }_{\HS}^2 = \tr( (Zf)(\pi)^* (Zg)(\pi) (Zg)(\pi)^* (Zf)(\pi) ) = \tr( (Zf)(\pi) (Zf)(\pi)^* (Zg)(\pi) (Zg)(\pi)^* ) \]
\[ = | \langle (Zg)(\pi) (Zg)(\pi)^*, (Zf)(\pi) (Zf)(\pi)^* \rangle_{\HS} | \leq \Norm{ (Zg)(\pi) (Zg)(\pi)^* }_{\HS} \Norm{ (Zf)(\pi) (Zf)(\pi)^* }_{\HS} \]
\[ = \Norm{ [g,g](\pi) }_{\HS} \Norm{ [f,f](\pi) }_{\HS}. \qedhere \]
\end{proof}

Equation \eqref{eq:innBrack} implies that vectors in $\H_\rho$ are uniquely determined by their bracket values. Specifically, if $f,g \in \H_\rho$ have $[f,h] = [g,h]$ for all $h \in \H_\rho$, then \eqref{eq:innBrack} shows that $\langle f,h \rangle = \langle g, h \rangle$, so that $f=g$. 
Propositions \ref{prop:brackProp1} and \ref{prop:brackProp2} together give the general feeling that the bracket behaves like a kind of operator-valued inner product on $\H_\rho$.\footnote{For representations of discrete groups, this idea was made more precise using the language of Hilbert modules and a slightly different notion of bracket in \cite{BHP3}.} However, the bracket can tell us about much more than the linear and geometric properties of $\H_\rho$. It can tell us about $\rho$ itself.

For each $\pi \in \hat{K}$, we will denote $\mathcal{M}_\pi$ for the isotypical component of $\pi$ in $\rho$. In other words, $\mathcal{M}_\pi$ is the closed linear span of all invariant subspaces of $\H_\rho$ on which $\rho$ is equivalent to $\pi$. We will write $P_\pi$ for the orthogonal projection of $\H_\rho$ onto $\mathcal{M}_\pi$. Finally, when $V \subset \H_\rho$ is an invariant subspace, we denote $\rho^V$ for the subrepresentation of $\rho$ on $V$. Then we have the following proposition.

\begin{prop} \label{prop:brackProp3} The bracket carries the following information about the isotypical components of $\rho$.
\begin{enumerate}[(i)]
\item For all $\pi \in \hat{K}$,
\[ \mathcal{M}_\pi = \{ f \in \H_\rho : [f,g](\sigma) = 0 \text{ for all }g\in \H_\rho \text{ and }\sigma \neq \overline{\pi}\} \]
\[ = \{ f \in \H_\rho : [f,f](\sigma) = 0 \text{ for }\sigma \neq \overline{\pi} \}. \]

\item For all $f, g \in \H_\rho$,
\[ [f,g](\overline{\pi}) = [ P_\pi f, g ](\overline{\pi}) \qquad (\pi \in \hat{K}). \]

\item For all $f \in \H_\rho$
\[ \rank [f,f](\pi)  = \mult(\overline{\pi}, \rho^{\langle f \rangle }) \qquad (\pi \in \hat{K}). \]
In particular,
\[ \dim \langle f \rangle = \sum_{\pi \in \hat{K}} d_\pi \cdot \rank [f,f](\pi). \]
\end{enumerate}
\end{prop}

\begin{proof}
As in the proof of the last proposition, we may assume that $K$ is a closed subgroup of a second countable locally compact group $G$, that $\H_\rho$ is a $K$-invariant subspace of $L^2(G)$, and that $\rho$ is left translation by $K$. If $Z$ is the Zak transform for the pair $(G,K)$, then the bracket is given by
\[ [f,g](\pi) = (Zg)(\pi)^* (Zf)(\pi) \qquad (f,g \in \H_\rho;\ \pi \in \hat{K}). \]
For any $f \in \H_\rho$ and $\pi \in \hat{K}$, this implies in particular that $(Zf)(\pi) = 0$ if and only if $[f,f](\pi) = 0$. Moreover, the Cauchy-Schwarz type inequality \eqref{eq:brackProp2.1} shows that $[f,f](\pi)=0$ if and only if $[f,g](\pi) = 0$ for all $g \in \H_\rho$. Now (i) follows from Proposition \ref{prop:isoComp}.

For (ii), apply Proposition \ref{prop:isoComp} to see that $(ZP_\pi f)(\overline{\pi}) = (Zf)(\overline{\pi})$. 

Finally, (iii) follows from \eqref{eq:multRan}, Theorem \ref{thm:ranGen}, and the fact that
\[ \rank [f,f](\pi) = \rank((Zf)(\pi)^* (Zf)(\pi)) = \rank( (Zf)(\pi) ) \qquad (\pi \in \hat{K}). \qedhere \]
\end{proof}

In many cases, statement (iii) above can be used to test whether a particular vector in $\H_\rho$ is cyclic for $\rho$.

\begin{prop}\label{prop:cycBrack}
Suppose that $\mult(\pi,\rho) < \infty$ for each $\pi \in \hat{K}$. Then $f\in \H_\rho$ is a cyclic vector for $\rho$ if and only if
\[ \rank [f,f](\pi) = \mult(\overline{\pi},\rho) \quad \text{for every }\pi \in \hat{K}. \]
Moreover, when $\dim \H_\rho < \infty$, $f$ is a cyclic vector if and only if
\[ \sum_{\pi \in \hat{K}} d_\pi \cdot \rank [f,f](\pi) = \dim \H_\rho. \]
\end{prop}

We can now prove our main result.

\begin{proof}[Proof of Theorem \ref{thm:brackFrm}]
By Lemma \ref{lem:cycIsom}, we may assume that $f$ is a function of positive type in $L^2(K)$, and that $\rho$ is given by left translation. We are going to apply Theorem \ref{thm:tranFrm} with $G=K$ and $\A = \{ f \}$. As explained in Remark \ref{rem:ZakFour}, the Zak transform reduces to the Fourier transform in this case. In particular, Theorem \ref{thm:ranGen} gives $\langle f \rangle = S(\A) = V_J$, where
\[ J(\pi) = \ran \hat{f}(\pi) \qquad (\pi \in \hat{K}). \]

It remains to show that our condition (ii) is equivalent to condition (ii) in Theorem \ref{thm:tranFrm}. 
For fixed $\pi \in \hat{K}$, we have $\hat{f}(\pi) \geq 0$, since $f$ is a function of positive type. Choose an orthonormal basis $e_1^\pi,\dotsc,e_{d_\pi}^\pi$ for $\H_\pi$ consisting of eigenvectors for $\hat{f}(\pi)$, with corresponding eigenvalues $\lambda_1^\pi \geq \dotsc \geq \lambda_{d_\pi}^\pi \geq 0$. If $r_\pi = \rank \hat{f}(\pi)$, then the nonzero eigenvalues of $[f,f](\pi) = \hat{f}(\pi)^2$ are precisely $(\lambda_1^\pi)^2,\dotsc, (\lambda_{r_\pi}^\pi)^2$. Now $\{\hat{f}(\pi) e_i^\pi\}_{i=1}^{d_\pi} = \{ \lambda_i^\pi e_i^\pi\}_{i=1}^{d_\pi}$ is a discrete frame for $J(\pi) = \spn\{e_1^\pi,\dotsc,e_{r_\pi}^\pi\}$ with bounds $A,B$ if and only if $A \leq (\lambda_1^\pi)^2,\dotsc, (\lambda_{r_\pi}^\pi)^2 \leq B$.
\end{proof}

\begin{example} \label{ex:irredFrm}
When $\rho$ is irreducible, it is well known that any nonzero $f \in \H_\rho$ generates a continuous tight frame with bound $\Norm{f}^2/ (\dim \H_\rho)$. We can recover this fact as follows. First, Proposition \ref{prop:brackProp3}(iii) shows that
\[ \rank [f,f](\pi) = \begin{cases}
1, & \text{if }\pi = \overline{\rho} \\
0, & \text{if }\pi \neq \overline{\rho}
\end{cases} \qquad (\pi \in \hat{K}). \]
In particular, the operators $[f,f](\pi)$, $\pi \in \hat{K}$, have only one nonzero eigenvalue between them. Call that eigenvalue $\lambda$. By Theorem \ref{thm:brackFrm},  $\{\rho(\xi) f\}_{\xi \in K}$ is a continuous tight frame with bound $\lambda$. Now use Proposition \ref{prop:brackProp2}(i) to compute $\Norm{f}^2 = \lambda\cdot (\dim \H_\rho)$.
\end{example}

\begin{example} \label{ex:dihedFrm}

Let $D_3 = \langle a, b : a^3 = b^2 = 1, bab^{-1} = a^{-1} \rangle$ be the dihedral group of order six. It has three irreducible representations: the trivial representation $\pi_1$, the one-dimensional representation $\pi_2$ given by $\pi_2(a) = 1$ and $\pi_2(b) = -1$, and the two-dimensional representation $\pi_3$ given by
\[ \pi_3(a) = \begin{pmatrix}
\omega & 0 \\
0 & \omega^{-1}
\end{pmatrix}
\quad \text{and} \quad
\pi_3(b) = \begin{pmatrix}
0 & 1 \\
1 & 0
\end{pmatrix}. \]
Consider the four-dimensional representation $\rho$ given by
\[ \rho(a) = \frac{1}{4} \begin{pmatrix}
1 & i\sqrt{3} & -3 & i\sqrt{3} \\
i\sqrt{3} & 1 & i\sqrt{3} & -3 \\
-3 &  i\sqrt{3} & 1 &  i\sqrt{3} \\
 i\sqrt{3} & -3 &  i\sqrt{3} & 1
\end{pmatrix}
\quad \text{and} \quad
\rho(b) = \frac{1}{2} \begin{pmatrix}
1 & 1 & 1 & -1 \\
1 & -1 & -1 & -1 \\
1 & -1 & 1 & 1 \\
-1 & -1 & 1 & -1
\end{pmatrix}. \]
Let $f = (3,1,-1,1)$. One can compute $[f,f](\pi_1) = 4$, $[f,f](\pi_2) = 4$, and
\[ [f,f](\pi_3) = \begin{pmatrix}
0 & 0 \\
0 & 2
\end{pmatrix}. \]
By the dimension count in Proposition \ref{prop:cycBrack}, $\langle f \rangle = \H_\rho = \C^4$. Applying Theorem \ref{thm:brackFrm}, we see that the orbit of $f$ forms a continuous frame for $\C^4$ with optimal bounds $2$ and $4$. When viewed as a discrete frame, the optimal bounds are $12$ and $24$. (See Remark \ref{rem:frmBnds}.)

As this example demonstrates, bracket analysis can result in significant dimension reduction for the study of group frames. Suppose, for instance, that we want to know the optimal frame bounds for $\{\rho(\xi) f\}_{\xi \in K}$. A naive approach to this problem would be to compute the Gramian operator for the sequence $\{\rho(x) f \}_{x \in K}$ and find the range of its nonzero eigenvalues. In this example, that would mean computing the eigenvalues of a $6\times 6$ matrix, which could be intractably difficult. Using bracket analysis, on the other hand, the largest matrix we had to analyze was $2 \times 2$.
\end{example}

\medskip

\section{Applications of bracket analysis} \label{sec:brackApp}

We now explore several applications of the bracket analysis developed in Section \ref{sec:brack}.

\subsection{Block diagonalization of the Gramian}

As we have just seen, the orbit $\{\rho(\xi) f\}_{\xi \in K}$ of a vector $f \in \H_\rho$ forms a frame only under special circumstances. However, compactness of $K$ implies that it is always a \emph{Bessel} mapping. Indeed, the Cauchy-Schwarz inequality produces
\[ \int_K | \langle g, \rho(\xi) f \rangle |^2 \, d\xi \leq \int_K \Norm{g}^2 \cdot \Norm{\rho(\xi) f}^2\, d\xi = \Norm{f}^2\cdot \Norm{g}^2 \qquad (g \in \H). \]
In particular, the Gramian $\mathcal{G} \colon L^2(K) \to L^2(K)$ and the frame operator $S \colon \H_\rho \to \H_\rho$ are well-defined for any choice of $f \in \H_\rho$, whether or not $\{\rho(\xi) f\}_{\xi \in K}$ is a frame.

A direct computation shows the Gramian is given by
\begin{equation} \label{eq:framOpGrm1}
\mathcal{G}(\phi) = \phi * V_f f \qquad (\phi \in L^2(K)),
\end{equation}
and the frame operator satisfies
\[ V_h (Sg) = V_f g * V_h f \qquad (g,h \in \langle f \rangle). \]
Thus, $S$ is defined uniquely by the relation
\begin{equation} \label{eq:framOpGrm2}
[Sg,h](\pi) = [f,h](\pi)\cdot [g,f](\pi) \qquad (g,h \in \langle f \rangle;\ \pi \in \hat{K}).
\end{equation}
The Gramian and the frame operator are intimately connected through the linear isometry $T_f \colon \langle f \rangle \to L^2(K)$ from Lemma \ref{lem:cycIsom}. Indeed, given any $g \in \langle f \rangle$, we compute
\[ (T_f Sg)\caret(\pi) = ( [f,f](\pi)^{1/2})^\dagger\cdot [Sg,f](\pi) = ( [f,f](\pi)^{1/2})^\dagger\cdot [f,f](\pi)\cdot [g,f](\pi) \]
\[ = [f,f](\pi)\cdot ( [f,f](\pi)^{1/2} )^\dagger \cdot [g,f](\pi) = (\mathcal{G}T_f g)\caret(\pi) \qquad (\pi \in \hat{K}). \]
Therefore,
\begin{equation} \label{eq:framOpGrm3}
T_f S = \mathcal{G} T_f.
\end{equation}

In fact, when $\{\rho(\xi) f\}_{\xi \in K}$ is a frame for $\langle f \rangle$, $T_f$ is the analysis operator for the canonical tight frame. To see this, first observe that the range of $T_f$ is $\langle T_f f \rangle$, the left ideal generated by $T_f f$. Let $R$ be the operator on $\ran T_f$ given by
\[ R(\phi) = \phi * (T_f f) \qquad (\phi \in \ran T_f ). \]
For any $g \in \langle f \rangle$, we have
\[ \langle g, \rho(\xi) f \rangle = \langle T_f g, L_\xi(T_f f) \rangle = [(T_f g)*(T_f f)^*](\xi) = [(T_f g)* (T_f f)](\xi) = (RT_f g)(\xi) \qquad (\xi \in K). \]
In other words, the analysis operator $T\colon \langle f \rangle \to L^2(K)$ for the frame $\{ \rho(\xi) f \}_{\xi \in K}$ is given by 
\[ T = R T_f. \]
Moreover, the computation above shows that $V_f f = (T_f f)*(T_f f)$, so 
\begin{equation} \label{eq:convFrmOp}
R^2 T_f = \mathcal{G} T_f = T_f S.
\end{equation}
The operator $R$ is positive semidefinite; for any $\phi \in \ran T_f$, we have
\[ \langle \phi, R(\phi) \rangle = \langle \phi, \phi * (T_f f) \rangle = \langle \phi^* * \phi, T_f f \rangle = \int_K (\phi^* * \phi)(\xi)\cdot \overline{(T_f f)(\xi)}\, d\xi \geq 0,\]
since $\overline{T_f f}$ is also a function of positive type. Since $T_f$ is a linear isometry, it follows from \eqref{eq:convFrmOp} that $T_f S^{1/2} = R T_f = T$. Equivalently, $T_f = T S^{-1/2}$, as desired.

\medskip

One is often interested in the spectrum $\sigma(\G)$ of the Gramian, since the optimal frame bounds are precisely the infimum and supremum of $\sigma(\G) \setminus\{0\}$. For a general positive semidefinite operator, finding the spectrum means diagonalization, which may be extremely difficult. For group frames, however, the realization of $\G$ as a convolution operator in \eqref{eq:framOpGrm1} can take us a long way in this direction, as in the proposition below.

\begin{prop} \label{prop:blkDiagGrm}
Fix any $f \in \H_\rho$, and let $\mathcal{G} \colon L^2(K) \to L^2(K)$ be the Gramian for the Bessel mapping $\{\rho(\xi) f\}_{\xi \in K}$. For each $\pi \in \hat{K}$, choose an orthonormal basis for $B(\H_\pi)$ with respect to the inner product $\langle A, B \rangle = d_\pi \tr(B^* A)$. Let $M_{[f,f](\pi)} \in M_{d_\pi^2}(\C)$ be the matrix over this basis for the operator $M_{[f,f](\pi)} \colon B(\H_\pi) \to B(\H_\pi)$ given by
\[ M_{[f,f](\pi)}(A) = [f,f](\pi)\cdot A. \]
If $\hat{K} = \{ \pi_1, \pi_2,\dotsc\}$, then $\G$ is unitarily equivalent to the block diagonal matrix
\[ \tilde{\mathcal{G}} = \begin{pmatrix}
M_{[f,f](\pi_1)} & & & 0 \\
& M_{[f,f](\pi_2)} && \\
0 && \ddots
\end{pmatrix}, \]
and the Fourier transform $\F \colon L^2(K) \to \bigoplus_{\pi \in \hat{K}} B(\H_\pi)$ is a conjugating unitary. That is, $\tilde{\mathcal{G}} = \F \mathcal{G} \F^{-1}$.
\end{prop}

\begin{proof}
This is obvious from the formulae \eqref{eq:framOpGrm1}, which gives the Gramian as a convolution operator, and \eqref{eq:FourConv}, which says the Fourier transform turns convolution operators into multiplication operators.
\end{proof}

Proposition \ref{prop:blkDiagGrm} leads to an alternative proof of Theorem \ref{thm:brackFrm}. Briefly: the spectrum of $\G$ is the union of the eigenvalues for $M_{[f,f](\pi)}$ as $\pi$ runs through $\hat{K}$, and the eigenvalues for $M_{[f,f](\pi)}$ are the same as those for $[f,f](\pi)$. Now use the fact that a Bessel mapping is a frame if and only if the nonzero elements of $\sigma(\G)$ are bounded away from zero, with the optimal frame bounds equal to the infimum and supremum of $\sigma(\G)\setminus\{0\}$, respectively.

\subsection{Classification of $K$-frames}

Continuous frames of the form $\{\rho(\xi) f\}_{\xi \in K}$ are sometimes called $K$-frames. We will say that $\rho$ \emph{admits} a $K$-frame if $\H_\rho$ has a continuous frame of this form. In that case, the orbit of $f$ spans $\H_\rho$, so in particular $\rho$ is cyclic. Greenleaf and Moskowitz \cite[Theorem 1.10]{GM} have reduced the property of being cyclic to a count of multiplicities of irreducible representations. Explicitly, they have shown that $\rho$ is cyclic if and only if $\mult(\pi,\rho) \leq d_\pi$ for each $\pi \in \hat{K}$. The following theorem refines this result for $K$-frames.

\begin{theorem} \label{thm:KFrmDim}
The following are equivalent.
\begin{enumerate}[(i)]
\item $\rho$ admits a $K$-frame.
\item $\rho$ admits a Parseval $K$-frame.
\item $\rho$ is cyclic, and $\dim \H_\rho < \infty$.
\item For all $\pi \in \hat{K}$, $\mult(\pi, \rho) \leq d_\pi$. Moreover, $\mult(\pi,\rho) = 0$ for all but finitely many $\pi \in \hat{K}$.
\end{enumerate}
\end{theorem}

The result of Greenleaf and Moskowitz mentioned above says, in part, that every subrepresentation of the regular representation of $K$ on $L^2(K)$ admits a cyclic vector. Theorem \ref{thm:KFrmDim} shows that this result can not be improved using the language of frames. In particular, the regular representation admits a $K$-frame if and only if $K$ is finite.

\begin{proof}
The equivalence of (iii) and (iv) is obvious from \cite[Theorem 1.10]{GM} and the formula
\[ \dim \H_\rho = \sum_{\pi \in \hat{K}} d_\pi\cdot \mult(\pi,\rho). \]
It remains to prove that (i), (ii), and (iv) are equivalent.

(i) $\implies$ (iv). Let $f \in \H_\rho$ be such that $\{\rho(\xi) f\}_{\xi \in K}$ is a continuous frame for $\H_\rho$, with lower frame bound $A > 0$. Since $\rho$ is cyclic, \cite[Theorem 1.10]{GM} shows that $\mult(\pi,\rho) \leq d_\pi$ for all $\pi \in \hat{K}$. By Proposition \ref{prop:brackProp2}(i), Theorem \ref{thm:brackFrm}, and Proposition \ref{prop:brackProp3}(iii),
\[ \Norm{f}^2 = \sum_{\pi \in \hat{K}} d_\pi \tr( [f,f](\pi) ) \geq \sum_{\pi \in \hat{K}} d_\pi A \cdot \rank([f,f](\pi)) = A \sum_{\pi \in \hat{K}} d_\pi \mult(\overline{\pi},\rho). \]
Consequently, $\mult(\pi,\rho) = 0$ for all except finitely many $\pi \in \hat{K}$.

(iv) $\implies$ (ii). We are going to embed $\H_\rho$ as a translation-invariant subspace of $L^2(K)$. Recalling that the Zak transform for the pair $(K,K)$ is the usual Fourier transform on $L^2(K)$ (see Remark \ref{rem:ZakFour}), we can then use the results of Section \ref{sec:ranTran} to analyze $\H_\rho$. 

For each $\pi \in \hat{K}$, choose a subspace $J(\pi) \subset \H_\pi$ of dimension equal to $\mult(\overline{\pi},\rho)$. Let
\[ V_J = \{ f \in L^2(K) : \ran \hat{f}(\pi) \subset J(\pi) \text{ for each }\pi \in \hat{K} \} \]
be the translation invariant subspace of $L^2(K)$ corresponding to the range function $J$. Since representations of $K$ are determined up to unitary equivalence by multiplicities of irreducible representations, we may assume by \eqref{eq:multRan} that $\H_\rho = V_J$, and that $\rho$ is given by left translation. For each $\pi \in \hat{K}$, let $P_\pi \in B(\H_\pi)$ be orthogonal projection onto $J(\pi)$. Then
\[ \sum_{\pi \in \hat{K}} d_\pi \Norm{P_\pi}_{\HS}^2 = \sum_{\pi \in \hat{K}} d_\pi \dim J(\pi) = \sum_{\pi \in \hat{K}} d_\pi \mult(\overline{\pi},\rho) < \infty, \]
so there is a function $f \in L^2(K)$ with $\hat{f}(\pi) = P_\pi$ for all $\pi \in \hat{K}$, by Plancherel's Theorem. Moreover, $\langle f \rangle = V_J = \H_\rho$ by Theorem \ref{thm:ranGen}. Finally, Lemma \ref{lem:ZakBrack} shows that
\[ [f,f](\pi) = \hat{f}(\pi)^* \hat{f}(\pi) = P_\pi \qquad (\pi \in \hat{K}), \]
so $\{\rho(\xi) f\}_{\xi \in K}$ is a continuous Parseval frame for $\H_\rho$, by Theorem \ref{thm:brackFrm}.

(ii) $\implies$ (i). This is trivial.
\end{proof}

Two $K$-frames $\{\rho(\xi)f\}_{\xi \in K}$ and $\{\rho'(\xi) f'\}_{\xi \in K}$ are \emph{unitarily equivalent} if there is a unitary $U\colon \H_\rho \to \H_{\rho'}$ such that $U \rho(\xi)f = \rho'(\xi) f$ for all $\xi \in K$. Equivalently, $U$ is a unitary equivalence of $\rho$ and $\rho'$ satisfying $U f = f'$. We now classify $K$-frames up to unitary equivalence.

In the theorem below, we treat $L^2(K)$ as a Banach $*$-algebra under convolution. Thus, a \emph{projection} in $L^2(K)$ is a function $f$ with the property that $f = f * f = f^*$. Equivalently, it is a function $f$ such that $\hat{f}(\pi)$ is an orthogonal projection for each $\pi \in \hat{K}$. We also write
\[ \mathcal{E}(K) = \{f \in L^2(K) : \hat{f}(\pi) = 0 \text{ for all but finitely many }\pi \in \hat{K}\} \]
for the space of trigonometric polynomials on $K$. Every projection in $L^2(K)$ belongs to $\mathcal{E}(K)$.

\begin{theorem} \label{thm:KFrmClas}
Up to unitary equivalence, $K$-frames are indexed by functions of positive type in $\mathcal{E}(K)$. If $f$ is such a function, the associated frame is $\{L_\xi f \}_{\xi \in K}$. The same correspondence sets up a bijection between equivalence classes of Parseval $K$-frames and projections in $L^2(K)$. 
\end{theorem}

In the special case where $K$ is finite, some aspects of this theorem appear implicitly in Vale and Waldron \cite{VW}. See also Han \cite{Ha}. For Parseval $K$-frames, the fact that the generating function $f$ is a projection implies that $V_f f = f * f^* = f$. By \eqref{eq:framOpGrm1}, the Gramian of the associated frame is the convolution operator $g \mapsto g * f$, which is orthogonal projection onto $\langle f \rangle$. In this sense, the theorem above may be compared with a result of Han and Larson \cite[Corollary 2.7]{HL}, which says that the correspondence between a frame and its Gramian induces a bijection between equivalence classes of Parseval frames indexed by a set $I$, and orthogonal projections on $\ell^2(I)$. For continuous frames, a similar result appears in \cite[Proposition 2.1]{GH}. Lots of orthogonal projections on $L^2(K)$ correspond to continuous Parseval frames over $K$. The projections that correspond to $K$-frames are precisely those given by convolution.

\begin{proof}
We use the term \emph{cyclic structure} for a pair $(\rho,f)$ consisting of a cyclic representation $\rho$ and a cyclic vector $f\in \H_\rho$. Call two cyclic structures $(\rho,f)$ and $(\rho',f')$ \emph{equivalent} if there is a unitary equivalence between $\rho$ and $\rho'$ that maps $f$ to $f'$. This agrees with the notion of equivalence of $K$-frames. Given $f\in L^2(K)$, we will denote $\rho_f$ for the subrepresentation of the regular representation on
\[ \langle f \rangle = \{ g \in L^2(K) : \ran \hat{g}(\pi) \subset \ran \hat{f}(\pi) \text{ for all }\pi \in \hat{K}\}. \]
Lemma \ref{lem:cycIsom} shows that
\[ \{ (\rho_f, f) : f \in L^2(K)\text{ is a function of positive type}\} \]
is a complete and irredundant set of cyclic structures, up to equivalence. For a fixed function $f \in L^2(K)$ of positive type, it only remains to show
\begin{equation} \label{eq:KFrmClas1}
\{L_\xi f \}_{\xi \in K} \text{ is a frame for }\langle f \rangle \iff \hat{f}(\pi) = 0 \text{ for all but finitely many }\pi \in \hat{K}
\end{equation}
and
\begin{equation} \label{eq:KFrmClas2}
\{L_\xi f\}_{\xi \in K} \text{ is a Parseval frame for }\langle f \rangle \iff \hat{f}(\pi) \text{ is an orthogonal projection for all }\pi \in \hat{K}.
\end{equation}

The forward implication of \eqref{eq:KFrmClas1} follows from Theorem \ref{thm:KFrmDim}, since 
\[ \mult(\overline{\pi}, \rho_f) = \rank \hat{f}(\pi) \qquad (\pi \in \hat{K}), \]
by \eqref{eq:multRan}. For the reverse implication, suppose that $\hat{f}(\pi) = 0$ for all but finitely many $\pi \in \hat{K}$. Then the operators $[f,f](\pi) = \hat{f}(\pi)^2$, $\pi \in \hat{K}$, have only finitely many nonzero eigenvalues between them, so $\{L_\xi f\}_{\xi \in K}$ is a continuous frame, by Theorem \ref{thm:brackFrm}.

To prove \eqref{eq:KFrmClas2}, recall that $\hat{f}(\pi) \geq 0$ for all $\pi \in \hat{K}$, so the eigenvalues of $[f,f](\pi) = \hat{f}(\pi)^2$ are precisely the squares of the eigenvalues of $\hat{f}(\pi)$. By Theorem \ref{thm:brackFrm}, $\{L_\xi f\}_{\xi \in K}$ is a continuous Parseval frame for $\langle f \rangle$ if and only if 0 and 1 are the only eigenvalues of $\hat{f}(\pi)$, $\pi \in \hat{K}$. Since the operators $\hat{f}(\pi)$ are self-adjoint, that happens if and only if each $\hat{f}(\pi)$ is an orthogonal projection.
\end{proof}

\begin{rem}
A function $f \in L^2(K)$ is a projection if and only if $\hat{f}(\pi)$ is an orthogonal projection for each $\pi \in \hat{K}$. If we let $J(\pi) = \ran \hat{f}(\pi) \subset \H_\pi$, we see that Parseval $K$-frames can also be classified by range functions in $\{ \H_\pi \}_{\pi \in \hat{K}}$ with the property that $J(\pi) = 0$ for all but finitely many $\pi \in \hat{K}$.
\end{rem}

Given a projection $f \in L^2(K)$, $\{L_\xi f\}_{\xi \in K}$ is a frame only for its closed linear span in $L^2(K)$, not necessarily for the whole space. This is troublesome in practice, where one usually wants coordinates for a frame in its ``native domain''. The corollary below gives such coordinates for every Parseval $K$-frame. When a matrix space $M_{m,n}(\C)$ is treated as a Hilbert space below, its inner product is gotten from the natural identification with $\C^{mn}$.

\begin{cor} \label{cor:nonAbelHarm}
For each $\pi \in \hat{K}$, choose an integer $r_\pi \in \{0,\dotsc,d_\pi\}$, in such a way that only finitely many $r_\pi \neq 0$. Choose an orthonormal basis for $\H_\pi$, and let $\pi_{i,j} \in C(K)$ be the corresponding matrix elements. Given $\xi \in K$, define $M_\xi(\pi) \in M_{r_\pi,d_\pi}(\C)$ by 
\[ M_\xi(\pi) = (\sqrt{d_\pi} \pi_{i,j}(\xi) )_{1\leq i \leq r_\pi, 1 \leq j \leq d_\pi}. \]
Then $\{M_\xi \}_{\xi \in K}$ is a continuous Parseval frame for $\bigoplus_{\pi \in \hat{K}} M_{r_\pi,d_\pi}(\C)$, and it is a $K$-frame when indexed $\{ M_{\xi^{-1}}\}_{\xi \in K}$. Up to unitary equivalence, every Parseval $K$-frame is produced in this way. 
\end{cor}

\begin{proof}
First we will show that $\{M_{\xi^{-1}}\}_{\xi \in K}$ is a Parseval $K$-frame. For each $\pi \in \hat{K}$, let $e_1^\pi,\dotsc,e_{d_\pi}^\pi$ be the orthonormal basis for $\H_\pi$ used in the construction of $\{M_\xi\}_{\xi \in K}$. Let $P_\pi \in B(\H_\pi)$ be orthogonal projection onto $\spn\{e_1^\pi,\dotsc,e_{r_\pi}^\pi\}$. By Plancherel's Theorem, there is a projection $f \in L^2(K)$ with $\hat{f}(\pi) = P_\pi$ for each $\pi \in \hat{K}$. We are going to map
\[ \langle f \rangle = \{ g \in L^2(K) : \ran \hat{g}(\pi) \subset \ran P_\pi \text{ for each }\pi \in \hat{K} \} \]
unitarily onto $\bigoplus_{\pi \in \hat{K}} M_{r_\pi,d_\pi}(\C)$ in a way that sends the Parseval $K$-frame $\{ L_\xi f \}_{\xi \in K}$ to $\{M_{\xi^{-1}}\}_{\xi \in K}$. 

For each $\pi \in \hat{K}$, assign $B(\H_\pi)$ the inner product $\langle A, B \rangle = d_\pi \langle A, B \rangle_{\HS}$, as in Plancherel's Theorem. There is a unitary $U_\pi \colon B(\H_\pi) \to M_{d_\pi}(\C)$ that replaces each operator with $\sqrt{d_\pi}$ times its matrix over the chosen basis. Let 
\[ U\colon L^2(K) \to \bigoplus_{\pi \in \hat{K}} M_{d_\pi}(\C) \]
be the unitary that follows the Fourier transform $\F \colon L^2(K) \to \bigoplus_{\pi \in \hat{K}} B(\H_\pi)$ by an application of $U_\pi$ in every coordinate $\pi \in \hat{K}$. Given $\xi \in K$, the translation identity \eqref{eq:FourTrans} shows that
\[ (L_\xi f)\caret(\pi) = P_\pi \pi(\xi^{-1}) \qquad (\pi \in \hat{K}), \]
so the $\pi$-th coordinate of $U(L_\xi f)$ is the $d_\pi \times d_\pi$ matrix with $M_{\xi^{-1}}(\pi)$ in the top $r_\pi$ rows and zeros in the bottom $d_\pi - r_\pi$ rows. Moreover,
\[ U \langle f \rangle = \{ (A_\pi)_{\pi \in \hat{K}} \in \bigoplus_{\pi \in \hat{K}} M_{d_\pi}(\C) : \text{for each $\pi \in \hat{K}$, $A_\pi$ has zeros in the bottom $d_\pi - r_\pi$ rows} \}. \]
Following $U$ with the natural identification
\[ U \langle f \rangle \cong \bigoplus_{\pi \in \hat{K}} M_{r_\pi,d_\pi}(\C) \]
gives the desired unitary of $\langle f \rangle$ onto $\bigoplus_{\pi \in \hat{K}} M_{r_\pi,d_\pi}(\C)$.

To see that \emph{every} Parseval $K$-frame is produced in this way, reverse the procedure above for an arbitrary projection $f \in L^2(K)$. For each $\pi \in \hat{K}$, let $P_\pi = \hat{f}(\pi)$, let $r_\pi = \rank P_\pi$, and choose an orthonormal basis $e_1^\pi,\dotsc,e_{d_\pi}^\pi$ for $\H_\pi$ in such a way that $\ran P_\pi = \spn\{ e_1^\pi, \dotsc, e_{r_\pi}^\pi \}$. The Parsevel $K$-frame $\{M_{\xi^{-1}}\}_{\xi \in K}$ produced with these parameters is unitarily equivalent to $\{L_\xi f\}_{\xi \in K}$ through the isometries constructed above.
\end{proof}

In the special case where $K$ is finite and \emph{abelian}, the frames described in Corollary \ref{cor:nonAbelHarm} are precisely the ``harmonic'' frames made by deleting rows from a discrete Fourier transform (DFT) matrix. (See \cite{VW2} for another proof that harmonic frames come from group actions.) While each finite abelian group can be used to make only finitely many Parseval frames in this way, a nonabelian group can make uncountably many inequivalent Parseval frames, since there are uncountably many projections in $L^2(K)$. (For finite groups, this was observed in \cite{VW}.) Moreover, it is often possible to make \emph{real} frames using nonabelian groups, as in the next example.

\begin{example}
Let $K=D_3$. Use notation as in Example \ref{ex:dihedFrm}. If we choose each $r_\pi$ to be as large as possible in Corollary \ref{cor:nonAbelHarm}, we obtain the following tight frame:
\[ \begin{pmatrix}
\begin{pmatrix}
1
\end{pmatrix}
&
\begin{pmatrix}
1
\end{pmatrix}
&
\begin{pmatrix}
1
\end{pmatrix}
&
\begin{pmatrix}
1
\end{pmatrix}
&
\begin{pmatrix}
1
\end{pmatrix}
&
\begin{pmatrix}
1
\end{pmatrix} \\[5 pt]

\begin{pmatrix}
\pmb{1}
\end{pmatrix}
&
\begin{pmatrix}
\pmb{1}
\end{pmatrix}
&
\begin{pmatrix}
\pmb{1}
\end{pmatrix}
&
\begin{pmatrix}
\pmb{-1}
\end{pmatrix}
&
\begin{pmatrix}
\pmb{-1}
\end{pmatrix}
&
\begin{pmatrix}
\pmb{-1}
\end{pmatrix} \\[5 pt]

\begin{pmatrix}
\pmb{\sqrt{2}} & \pmb{0} \\
0 & \sqrt{2}
\end{pmatrix}
&
\begin{pmatrix}
\pmb{ \omega \sqrt{2} } & \pmb{0} \\
0 & \omega^2 \sqrt{2}
\end{pmatrix}
&
\begin{pmatrix}
\pmb{ \omega^2 \sqrt{2} } & \pmb{0} \\
0 & \omega \sqrt{2}
\end{pmatrix}
&
\begin{pmatrix}
\pmb{0} & \pmb{ \sqrt{2} } \\
\sqrt{2} & 0
\end{pmatrix}
&
\begin{pmatrix}
\pmb{0} & \pmb{ \omega\sqrt{2} } \\
\omega^2 \sqrt{2} & 0
\end{pmatrix}
&
\begin{pmatrix}
\pmb{0} & \pmb{ \omega^2 \sqrt{2} } \\
\omega \sqrt{2} & 0 
\end{pmatrix}
\end{pmatrix}. \]
We can get another tight frame by deleting some of the rows:
\[ \begin{pmatrix}

\begin{pmatrix}
\pmb{1}
\end{pmatrix}
&
\begin{pmatrix}
\pmb{1}
\end{pmatrix}
&
\begin{pmatrix}
\pmb{1}
\end{pmatrix}
&
\begin{pmatrix}
\pmb{-1}
\end{pmatrix}
&
\begin{pmatrix}
\pmb{-1}
\end{pmatrix}
&
\begin{pmatrix}
\pmb{-1}
\end{pmatrix} \\[5 pt]

\begin{pmatrix}
\pmb{\sqrt{2}} & \pmb{0}
\end{pmatrix}
&
\begin{pmatrix}
\pmb{ \omega \sqrt{2} } & \pmb{0}
\end{pmatrix}
&
\begin{pmatrix}
\pmb{ \omega^2 \sqrt{2} } & \pmb{0}
\end{pmatrix}
&
\begin{pmatrix}
\pmb{0} & \pmb{ \sqrt{2} }
\end{pmatrix}
&
\begin{pmatrix}
\pmb{0} & \pmb{ \omega\sqrt{2} }
\end{pmatrix}
&
\begin{pmatrix}
\pmb{0} & \pmb{ \omega^2 \sqrt{2} }
\end{pmatrix}
\end{pmatrix}. \]
This corresponds to choosing $r_1=0$ and $r_2=r_3 = 1$.
 Collapsing the interior matrices gives a tight frame for $\C^3$:
\[ \begin{pmatrix}
1 & 1 & 1 & -1 & -1 & -1 \\
\sqrt{2} & \omega\sqrt{2} & \omega^2 \sqrt{2} & 0 & 0 & 0 \\
0 & 0 & 0 & \sqrt{2} & \omega \sqrt{2} & \omega^2 \sqrt{2} \\
\end{pmatrix}. \]
The frame bound is $\card(D_3) = 6$; see Remark \ref{rem:frmBnds}.

Representing the two-dimensional representation over a different basis gives a completely different frame. If we use
\[ \pi_3(a) = \frac{1}{2} \begin{pmatrix}
-1 & -\sqrt{3} \\
 \sqrt{3} & -1
\end{pmatrix} \quad \text{and} \quad
\pi_3(b) = \begin{pmatrix}
1 & 0 \\
0 & -1
\end{pmatrix} \]
and choose rows exactly as above, we obtain the tight frame
\[ \begin{pmatrix}
1 & 1 & 1 & -1 & -1 & -1 \\
\sqrt{2} & -1/\sqrt{2} & -1/\sqrt{2} & \sqrt{2} & -1/\sqrt{2} & -1/\sqrt{2} \\
0 & -\sqrt{3/2} & \sqrt{3/2} & 0 & \sqrt{3/2} & -\sqrt{3/2}
\end{pmatrix}. \]
This time we used real representations, so we got a tight frame for $\R^3$.
\end{example}

\subsection{Disjointness properties}

Let $\H$ and $\K$ be separable Hilbert spaces carrying frames $\Phi=\{f_i\}_{i\in I}$ and $\Psi=\{g_i\}_{i\in I}$, respectively. 
We say that $\Phi$ and $\Psi$ are \emph{disjoint} if $\{ (f_i,g_i) \}_{i \in I}$ is a frame for $\H \oplus \K$. Disjoint frames were introduced independently by Balan \cite{Ba} and by Han and Larson \cite{HL}. For a detailed study of disjoint \emph{continuous} frames, see \cite{GH}. The corollary 
below says that $K$-frames from distinct isotypical components of $\rho$ are always disjoint, and that every $K$-frame can be decomposed into disjoint frames in this way. This will be generalized for group frames with multiple generators in  Corollary \ref{cor:multCompFrm}. Recall that $\mathcal{M}_\pi \subset \H_\rho$ denotes the isotypical component for $\pi \in \hat{K}$, and that $P_\pi \in B(\H_\rho)$ is orthogonal projection of $\H_\rho$ onto $\mathcal{M}_\pi$.

\begin{cor} \label{cor:compFrm}
Fix a vector $f \in \H_\rho$ and constants $A$ and $B$ with $0 < A \leq B < \infty$. The following are equivalent.
\begin{enumerate}[(i)]
\item $\{\rho(\xi) f\}_{\xi \in K}$ is a continuous frame for $\H_\rho$ with bounds $A, B$.
\item For each $\pi \in \hat{K}$, $\{ \rho(\xi) P_\pi f\}_{\xi \in K}$ is a continuous frame for $\mathcal{M}_\pi$ with bounds $A,B$.
\end{enumerate}
\end{cor}

\begin{proof}
For each $\pi \in \hat{K}$ and each $g \in \H_\rho$, Proposition \ref{prop:brackProp3} shows that
\begin{equation} \label{eq:compFrm1}
[P_\pi f, P_\pi g](\sigma) = [P_\pi f, g](\sigma) = \begin{cases}
[f,g](\overline{\pi}), & \text{if }\sigma = \overline{\pi} \\
0, & \text{if }\sigma \neq \overline{\pi}.
\end{cases}
\end{equation}
Taking $g = f$ above, we see that $f$ satisfies condition (ii) of Theorem \ref{thm:brackFrm} if and only if each $P_\pi f$ does the same. It remains to show that $\langle f \rangle = \H_\rho$ if and only if $\langle P_\pi f \rangle = \mathcal{M}_\pi$ for each $\pi \in \hat{K}$.

If $\langle f \rangle \neq \H_\rho$, then we can find a nonzero vector $g \in \H_\rho$ with $[f,g] =0$, by Proposition \ref{prop:brackProp1}(vi). Find $\pi \in \hat{K}$ for which $P_\pi g \neq 0$. Then \eqref{eq:compFrm1} shows that $[P_\pi f, P_\pi g] = 0$, so that $P_\pi g \perp \langle P_\pi f \rangle$. Thus, $\langle P_\pi f \rangle \neq \mathcal{M}_\pi$.

Conversely, if there is some $\pi \in \hat{K}$ for which $\langle P_\pi f \rangle \neq \mathcal{M}_\pi$, then there is a nonzero vector $g \in \mathcal{M}_\pi$ with $0 = [P_\pi f, g] = [f, P_\pi g] = [f,g]$. Hence, $g \perp \langle f \rangle$, and $\langle f \rangle \neq \H_\rho$.
\end{proof}

Recall that $\rho$ is \emph{multiplicity free} when all of its isotypical components are irreducible. Equivalently, this means that $\mult(\pi,\rho) \in \{0,1\}$ for all $\pi \in \hat{K}$. Corollary \ref{cor:compFrm} leads to an extension of Example \ref{ex:irredFrm} for multiplicity free representations.

\begin{cor} \label{cor:multFrFrm}
Suppose $\rho$ is multiplicity free. Let $E = \{ \pi \in \hat{K} : \mult(\pi,\rho) \neq 0\}$. For a nonzero vector $f \in \H_\rho$, the following are equivalent.
\begin{enumerate}[(i)]
\item $\{ \rho(\xi) f\}_{\xi \in K}$ is a tight frame for $\H_\rho$.
\item For any $\pi,\sigma \in E$, $\Norm{P_\pi f}^2/d_\pi = \Norm{P_\sigma f}^2/d_\sigma$.
\end{enumerate}
When this happens, the optimal frame bound is the common value of $\Norm{P_\pi f}^2/d_\pi$ for $\pi \in E$.
\end{cor}

In the special case where $K$ is \emph{finite}, the equivalence of (i) and (ii) above can be deduced from \cite[Theorem 6.18]{VW2}. 

We mention just one of a myriad applications for Corollary \ref{cor:compFrm}. An action of a group $G$ on a set $X$ is called \emph{2-transitive} when the following holds: for every two pairs $(x,y), (w,z) \in X \times X$ with $x \neq y$ and $w \neq z$, there is a single group element $g \in G$ with $g\cdot x = w$ and $g\cdot y = z$.

\begin{cor} \label{cor:2tranFrm}
Let $G$ be a finite group acting on a finite set $X$ with an action that is $2$-transitive. Fix a nonzero vector $f =(f_x)_{x\in X} \in \ell^2(X)$. Then $\{ (f_{g\cdot x})_{x\in X} : g \in G\}$ is a tight frame for $\ell^2(X)$ if and only if
\begin{equation} \label{eq:2tranFrm}
\left| \sum_{x\in X} f_x \right|^2 = \sum_{x\in X} | f_x |^2.
\end{equation}
\end{cor}

\begin{proof}
The statement is trivial when $X$ is a singleton, so we may assume that $X$ has more than one point. Let $\rho$ be the unitary representation of $G$ on $\ell^2(X)$ associated with the action of $G$. Namely, for $g \in G$ and $\psi = (\psi_x)_{x\in X} \in \ell^2(X)$, we define $\rho(g)\psi = (\psi_{g^{-1}\cdot x})_{x\in X}$. By \cite[Corollary 29.10]{JL}, $\rho$ is multiplicity free with two isotypical components,
\[ \mathcal{M}_1 = \left\{ (\psi_x)_{x\in X} \in \ell^2(X) : \psi_x = \psi_y \text{ for all }x,y \in X \right\} \]
and
\[ \mathcal{M}_2 = \left\{ (\psi_x)_{x\in X} \in \ell^2(X) : \sum_{x\in X} \psi_x = 0 \right\}. \]
Let $P_j$ be orthogonal projection of $\ell^2(X)$ onto $\mathcal{M}_j$, for $j=1,2$. If we denote 
\[ \overline{f} = \frac{1}{|X|} \sum_{x\in X} f_x, \]
then $P_1 f = (\overline{f})_{x\in X}$, and $P_2 f = (f_x - \overline{f})_{x\in X}$. In particular,
\[ \Norm{P_1 f}^2 = \frac{1}{|X|} \left| \sum_{x\in X} f_x \right|^2. \]
By Corollary \ref{cor:multFrFrm}, the orbit of $f$ under $\rho$ is a tight frame for $\ell^2(X)$ if and only if $\Norm{P_1 f}^2 = \Norm{P_2 f}^2 / (|X| - 1)$, if and only if $|X|\cdot \Norm{ P_1 f}^2 = \Norm{P_2 f}^2 + \Norm{P_1 f}^2 = \Norm{f}^2$, if and only if
\[ \left| \sum_{x \in X} f_x \right|^2 = \sum_{x\in X} | f_x |^2. \qedhere \]
\end{proof}

The proof indicates a simple and universal method for constructing the generating vector $f$. Let $\varphi \in \ell^2(X)$ be the all-ones vector. Fix any nonzero vector $\psi \in \ell^2(X)$ with $\sum_{x\in X} \psi_x = 0$, and scale it so that $\Norm{\psi}^2 = |X|^2 - |X|$. Then $f = \varphi + \psi$ generates a tight frame for $\ell^2(X)$, by Corollary \ref{cor:multFrFrm}. Up to scaling, every vector satisfying \eqref{eq:2tranFrm} is produced in this way.

\begin{example} \label{ex:PermFrm}
The action of the symmetric group $S_n$ on the set with $n$ elements is $2$-transitive. Thus, Corollary \ref{cor:2tranFrm} and the comment above explain how to make a unit norm tight frame of $n!$ vectors in $\C^n$ just by permuting the entries of a single vector.
\end{example}

\medskip

\section{Group frames with multiple generators} \label{sec:multGen}

The last two sections focused on frames generated by a single vector $f\in \H_\rho$. We now consider frames with \emph{multiple} generators. For a countable family $\A = \{f_j\}_{j\in I} \subset \H_\rho$, this means that we will determine precise (and simple) conditions under which the orbit $\{ \rho(\xi) f_j\}_{j\in I, \xi \in K}$ forms a continuous frame for $\H_\rho$. In the course of doing so, we will classify the invariant subspaces of $\H_\rho$ in terms of range functions. 

Despite significant interest in the problem, very little has been done in the area of group frames with multiple generators. The most fruitful area has been frames generated by translations, mostly with abelian groups \cite{BHP2,B,BR,CP,I,KR} but in at least one case with nonabelian \cite{CMO}. In the setting of discrete nonabelian groups, Hern{\'a}ndez and his collaborators \cite{BHP3} have recently developed an abstract machinery to handle frames with multiple generators for a special class of unitary representations. For finite groups and \emph{tight} frames, Vale and Waldron \cite{VW3} recently broke through the single generator barrier, with a neat condition in terms of norms and orthogonality of the generating vectors. These few papers provide the state of the art. 

Our main result is a duality theorem unifying the work of Vale and Waldron with classical duality of frames and Riesz sequences, simultaneously extending their results to non-tight frames and actions by compact groups. Here we pull ahead of the abelian setting. As far as the author knows, there is nothing of this kind in the literature for LCA groups. Once again, we hope that by illuminating the situation for nonabelian compact groups, we can set a path for further research on representations of general locally compact groups.

\medskip

Our notation and assumptions are as follows. Let $K$ and $\rho$ be as in the previous sections. Since $K$ is compact, it is always possible to decompose $\H_\rho$ as a direct sum of irreducible invariant subspaces. Our main assumption is that \emph{this has already been done}. For $\pi \in \hat{K}$, we let $m_\pi = \mult(\pi,\rho)$. We write $\pi^{\oplus m_\pi}$ for the direct sum of $m_\pi$ copies of $\pi$, which acts on $\H_\pi^{\oplus m_\pi}$. Without loss of generalty, we may assume that
\[ \rho = \bigoplus_{\pi \in \hat{K}} \pi^{\oplus m_\pi}, \]
and that
\[ \H_\rho = \bigoplus_{\pi \in \hat{K}} \H_\pi^{\oplus m_\pi}. \]
We warn that some of the multiplicities $m_\pi$ may be infinite, but since $\H_\rho$ is separable, they must all be countable.

Fix the following notation. Let $\A = \{f_j\}_{j\in I} \subset \H_\rho$ be a countable family of vectors. We write $f_j = (f_j^\pi)_{\pi \in \hat{K}} \in \H_\rho$, with $f_j^\pi = (f_{i,j}^\pi)_{i=1}^{m_\pi} \in \H_\pi^{\oplus m_\pi}$. We also denote
\[ E(\A) = \{ \rho(\xi) f_j\}_{j \in I, \xi \in K} \]
for the orbit of $\A$ under $\rho$. Formally, $E(\A)$ should be interpreted as a set with multiplicities, or more accurately, as a mapping $I \times K \to \H_\rho$. Finally, we let
\[ S(\A) = \overline{\spn}\{ \rho(\xi) f_j : j \in I, \xi \in K\} \]
be the invariant subspace generated by $\A$.

Our notation is meant to suggest that $\A$ is a kind of matrix. For each $\pi \in \hat{K}$, we define
\[ \A(\pi) = ( f_{i,j}^\pi )_{1 \leq i \leq m_\pi, j \in I}, \]
which is a (possibly infinite) matrix with entries in $\H_\pi$. The number of rows equals $m_\pi$, and the number of columns equals $\card(\A)$. For instance, if $\A$ were finite with $I = \{1,\dotsc, N\}$, we would have
\[ \A(\pi) = \begin{pmatrix}
| & | & \dots & | \\
f_1^\pi & f_2^\pi & \dotso & f_N^\pi \\
| & | & \dots & | 
\end{pmatrix}. \] 
If we now imagine the matrices $\A(\pi)$ stacked vertically, then the $j$-th column of the resulting ``matrix'' precisely describes the direct sum decomposition of $f_j \in \A$.

We remind the reader that a \emph{Riesz sequence} in a Hilbert space $\H$ is a sequence of vectors $\{ f_i \}_{i \in J} \subset \H$ for which there are constants $0 < A \leq B < \infty$ such that, whenever $(c_i) \in \ell^2(J)$ has finite support,
\[ A \sum_{i \in J} | c_i |^2 \leq \Norm{ \sum_{i \in J} c_i f_i }^2 \leq B \sum_{i \in J} |c_i|^2. \]
Once this inequality holds for those $(c_i) \in \ell^2(J)$ with finite support, it automatically holds for arbitrary $(c_i) \in \ell^2(J)$. Our main result, below, says that the frame properties of the orbit of the ``columns''  of $\A$ can be read from the Riesz properties of the \emph{rows}.

\begin{theorem} \label{thm:multFrm2}
The following are equivalent for constants $A$ and $B$ with $0 < A \leq B < \infty$.
\begin{enumerate}[(i)]
\item The orbit $E(\A) = \{\rho(\xi) f_j\}_{j \in I, \xi \in K}$ is a continuous frame for $\H_\rho$ with bounds $A,B$.
\item For every $\pi \in \hat{K}$, the rows of $\A(\pi)$ belong to $\H_\pi^{\oplus I}$, where they form a Riesz sequence with bounds $d_\pi A,d_\pi B$.
\end{enumerate}
\end{theorem}

This will actually be a corollary of a more general theorem. Theorem \ref{thm:multFrm} (infra) gives conditions for $E(\A)$ to form a continuous frame for a general invariant subspace of $\H_\rho$.

\begin{example}
Here are four special cases of Theorem \ref{thm:multFrm2}.

\smallskip

(1) When $K$ is the trivial group and $\rho$ is the trivial action of $K$ on $\C$, we recover the usual duality theorem for frames and Riesz sequences, which says that the columns of a matrix $M \in M_{m,n}(\C)$ form a frame for $\C^m$ if and only if the rows of $M$ form a Riesz sequence in $\C^n$. Moreover, the bounds of the frame and the Riesz sequence are the same.

\smallskip

(2) When $\A$ has a single vector $f$ and $\rho$ is irreducible, there is only one matrix $\A(\pi)$ to consider, namely $\A(\rho) = (f)$. Obviously its rows form a Riesz sequence with upper and lower bounds both equal to $\Norm{f}^2$, so the orbit $\{ \rho(\xi) f \}_{\xi \in K}$ is a tight frame for $\H_\rho$ with bound $\Norm{f}^2/(\dim \H_\rho)$. This is the conclusion of Example \ref{ex:irredFrm}.

\smallskip

(3) More generally, when $\rho$ is multiplicity free, we can easily recover Corollary \ref{cor:multFrFrm}.

\smallskip

(4) Taking $A=B$ in Theorem \ref{thm:multFrm2}, we see that $E(\A)$ is a tight frame for $\H_\rho$ with bound $A$ if and only if the rows of each matrix $\A(\pi)$ form an orthogonal sequence of vectors in $\H_\pi^{\oplus I}$, with each vector's norm equal to $\sqrt{d_\pi A}$. That is,
\[ \sum_{j \in I} \langle f_{i_1,j}^\pi, f_{i_2,j}^\pi \rangle = \delta_{i_1,i_2}\cdot d_\pi A. \]
In the case where $K$ and $\A$ are both \emph{finite}, this is a result of Vale and Waldron \cite[Theorem 2.8]{VW3}.
\end{example}

\begin{rem}
Neither the group $K$ nor the representation $\rho$ play a prominent role in condition (ii) of Theorem \ref{thm:multFrm2}, except to provide conditions on the direct sum decomposition $\H_\rho = \bigoplus_{\pi \in \hat{K}} \H_\pi^{\oplus m_\pi}$. Suppose, then, that $G$ is \emph{another} compact group acting on $\H_\rho$ with a representation $\eta$ that admits the same decomposition of $\H_\rho$ as a direct sum of irreducible invariant subspaces. Then the orbit of $\A$ under the action of $\rho$ is a frame for $\H_\rho$ if and only if the orbit under the action of $\eta$ is, too. Moreover, the frame bounds are the same in both cases.

While this may seem surprising at first, it is really an extension of a well-known phenomenon. After all, any nonzero vector $f\in \H_\rho$ generates a tight frame when $\rho$ acts irreducibly, and this mild condition $(f\neq 0)$ has nothing to do with $K$ or the particular irreducible representation $\rho$. As we have seen, this is a special case of Theorem \ref{thm:multFrm2}.
\end{rem}

\medskip

\subsection{Classification of invariant subspaces}

From a technical perspective, we can always find an encompassing group $G \supset K$ for which $\H_\rho$ embeds into $L^2(G)$ as a $K$-invariant subspace, with $\rho$ turning into left translation. (See Theorem \ref{thm:isom}.) In this sense, Theorem \ref{thm:tranFrm} on frames generated by translations already gives a complete characterization of group frames with multiple generators. In practice, however, it may be tedious to unravel this characterization through the embedding $\H_\rho \to L^2(G)$. Instead of following that route, we will now try to recreate the program of  Sections \ref{sec:Zak}--\ref{sec:tranFrm} from scratch. Namely, we will give a range function characterization of the invariant subspaces of $\H_\rho$, and then we will use that characterization to deduce Theorem \ref{thm:multFrm2}.

To begin our program, we need a substitute for the Zak transform. Fix $\pi \in \hat{K}$, and associate each sequence $\Phi = {(\phi_i)_{i=1}^{m_\pi}} \in {\H_\pi^{\oplus m_\pi}}$ with its analysis operator $T_\pi \Phi \colon \H_\pi \to \ell^2_{m_\pi}$, which is given by
\[ [ T_\pi \Phi ](\psi) = ( \langle \psi, \phi_i \rangle )_{i=1}^{m_\pi} \qquad (\psi \in \H_\pi). \]
Then $T_\pi \colon \H_\pi^{\oplus m_\pi} \to \HS( \H_\pi, \ell^2_{ m_\pi } )$ is a \emph{conjugate}-linear unitary. To see this, consider the composition of isomorphisms
\[ \H_\pi^{\oplus m_\pi} \cong \H_\pi \otimes \ell^2_{m_\pi} \cong \HS(\H_\pi, \ell^2_{m_\pi} ), \]
the last of which is conjugate linear (see \cite[Section 7.3]{F}). Letting $\pi$ run through $\hat{K}$, we obtain a conjugate-linear unitary
\[ T \colon \H_\rho \to \bigoplus_{\pi \in \hat{K}} \HS( \H_\pi, \ell^2_{m_\pi} ) \]
given by
\[ T (g_\pi)_{\pi \in \hat{K}} = ( T_\pi g_\pi )_{\pi \in \hat{K}} \qquad ( (g_\pi)_{\pi \in \hat{K}} \in \bigoplus_{\pi \in \hat{K}} \H_\pi^{\oplus m_\pi}  = \H_\rho ). \]
If we write $g_\pi = (g_i^\pi)_{i=1}^{m_\pi} \in \H_\pi^{\oplus m_\pi}$, then the simple formula $\langle \phi, \pi(\xi) g_i^\pi \rangle = \langle \pi(\xi^{-1}) \phi, g_i^\pi \rangle$ gives the key identity
\begin{equation}\label{eq:TRho}
(T \rho(\xi) g)(\pi) =  (Tg)(\pi)\cdot \pi(\xi^{-1}) \qquad (g \in \H_\rho,\ \xi \in K,\ \pi \in \hat{K}).
\end{equation}
This will serve as our substitute for the Zak transform's translation property \eqref{eq:ZakTrans}.

\medskip

A careful reading of Section \ref{sec:ranTran} shows that we used only two properties of the Zak transform: the translation property \eqref{eq:ZakTrans}, and the fact that $Z$ is unitary. In the current setting, we can therefore leverage the intertwining property \eqref{eq:TRho} to classify invariant subspaces of $\H_\rho$ in terms of range functions. Let $J$ be a range function in $\{ \ell^2_{m_\pi} \}_{\pi \in \hat{K}}$, and let 
\[ V_J = \{ (g_\pi)_{\pi \in \hat{K}} \in \H_\rho : \text{for each $\pi \in \hat{K}$, } \ran T_\pi g_\pi \subset J(\pi) \}. \]
Equivalently,
\[ T V_J = \bigoplus_{\pi \in \hat{K}} \HS(\H_\pi, J(\pi) ). \]
By \eqref{eq:TRho}, $V_J$ is an invariant subspace of $\H_\rho$. In fact, a trivial modification of the proof of Theorem \ref{thm:ranTran} shows that every invariant subspace of $\H_\rho$ takes this form. Explicitly, we have the following.

\begin{theorem}
The mapping $J \mapsto V_J$ is a bijection between range functions in $\{ \ell^2_{m_\pi} \}_{\pi \in \hat{K}}$ and invariant subspaces of $\H_\rho$.
\end{theorem}

In further analogy with the range function analysis of Section \ref{sec:ranTran}, it is easy to see that the correspondence $J \mapsto V_J$ preserves direct sum decompositions. This leads to the following analogue of Theorem \ref{thm:ranDecomp}.

\begin{theorem} \label{thm:ranDecompGen}
Let $J$ be a range function in $\{ \ell^2_{m_\pi}\}_{\pi \in \hat{K}}$. Choose an orthonormal basis $\{e_i^\pi\}_{i\in I_\pi}$ for each $J(\pi)$, $\pi \in \hat{K}$. For each $\pi \in \hat{K}$ and $i \in I_\pi$, let $V_{\pi,i}$ be the space of $(g_\sigma)_{\sigma \in \hat{K}} \in \H_\rho$ such that $g_\sigma = 0$ for $\sigma \neq \pi$, and such that $g_\pi = (g_j^\pi)_{j=1}^{m_\pi}$ satisfies $(\langle \phi,g_j^\pi\rangle )_{j=1}^{m_\pi} = c_\phi e_i^\pi$ for every $\phi \in \H_\pi$, where $c_\phi$ is a scalar. Then $V_{\pi,i}$ is an irreducible invariant subspace of $\H_\rho$, and
\[ V_J = \bigoplus_{\pi \in \hat{K}} \bigoplus_{i\in I_\pi} V_{\pi,i}. \]
Moreover, every decomposition of $V_J$ as a direct sum of irreducible subspaces occurs in this way.
\end{theorem}

When $J$ is the range function with $J(\pi) = \ell^2_{m_\pi}$ for every $\pi \in \hat{K}$, the theorem above describes every possible decomposition of $\H_\rho$ as a direct sum of irreducibles. Remember that our operating assumption is that we can find \emph{one} such decomposition. Thus, knowing one decomposition is enough to describe them all (and very simply, at that).

\subsection{Duality for frames with multiple generators}

Now we can prove our main theorem on group frames with multiple generators. Remember our interpretation of $\A$ as a kind of matrix, with the vectors $f_j \in \A$ appearing as the ``columns''. It turns out that the frame properties of the orbit of the ``columns'' of $\A$ can be read from a Riesz-like property on the \emph{rows}.

\begin{theorem} \label{thm:multFrm}
Let $J$ be a range function in $\{ \ell^2_{m_\pi} \}_{\pi \in \hat{K}}$, and assume that $\A \subset V_J$. For constants $A$ and $B$ with $0 < A \leq B < \infty$, the following are equivalent.
\begin{enumerate}[(i)]
\item $E(\A)$ is a continuous frame for $V_J$ with bounds $A,B$. That is,
\[ A \Norm{g}^2 \leq \sum_{j\in I} \int_K | \langle g, \rho(\xi) f_j \rangle |^2\, d\xi \leq B \Norm{g}^2 \qquad (g \in V_J). \]
\item For every $\pi \in \hat{K}$ and every sequence $(c_i)_{i=1}^{m_\pi} \in J(\pi) \subset \ell^2_{m_\pi}$,
\[ d_\pi A \sum_{i=1}^{m_\pi} |c_i|^2 \leq \sum_{j\in I} \Norm{ \sum_{i=1}^{m_\pi} c_i f_{i,j}^\pi }^2 \leq d_\pi B \sum_{i=1}^{m_\pi} |c_i|^2. \]
\end{enumerate}
\end{theorem}

\begin{proof}
Fix $g,h \in \H_\rho$. We will denote $g = (g_\pi)_{\pi \in \hat{K}}$, with $g_\pi \in \H_\pi^{\oplus m_\pi}$, and $g_\pi = (g_i^\pi)_{i=1}^{m_\pi}$, with $g_i^\pi \in \H_\pi$. We use a similar notation for $h$. For each $\pi \in \hat{K}$, fix an orthonormal basis $e_1^\pi,\dotsc,e_{d_\pi}^\pi$ for $\H_\pi$, and let $\pi_{i,j} \in C(K)$ be the corresponding matrix elements. We are going to decompose $V_h g \in L^2(K)$ in the orthonormal basis $\{ \sqrt{d_\pi} \pi_{i,j} : \pi \in \hat{K}, 1 \leq i,j \leq d_\pi\}$. 

For any $\xi \in K$, we can use \eqref{eq:TRho} and the fact that $T$ is a conjugate-linear unitary to write
\[ \langle g, \rho(\xi) h \rangle = \langle T \rho(\xi) h, T g \rangle = \sum_{\pi \in \hat{K}} \langle (T_\pi h_\pi) \pi(\xi^{-1}), T_\pi g_\pi \rangle_{\HS} = \sum_{\pi \in \hat{K}} \sum_{k=1}^{d_\pi} \langle (T_\pi h_\pi) \pi(\xi^{-1}) e_k^\pi, (T_\pi g_\pi) e_k^\pi \rangle \]
\[ = \sum_{\pi \in \hat{K}} \sum_{k=1}^{d_\pi} \sum_{i=1}^{m_\pi} \langle \pi(\xi^{-1}) e_k^\pi, h_i^\pi \rangle \langle g_i^\pi, e_k^\pi \rangle =  \sum_{\pi \in \hat{K}} \sum_{k=1}^{d_\pi} \sum_{i=1}^{m_\pi} \sum_{l=1}^{d_\pi} \langle e_l^\pi, h_i^\pi \rangle \langle \pi(\xi^{-1}) e_k^\pi, e_l^\pi \rangle \langle g_i^\pi, e_k^\pi \rangle \]
\[ =  \sum_{\pi \in \hat{K}} \sum_{k=1}^{d_\pi} \sum_{i=1}^{m_\pi} \sum_{l=1}^{d_\pi} \langle e_l^\pi, h_i^\pi \rangle \langle g_i^\pi, e_k^\pi \rangle \overline{\pi_{k,l}(\xi)}. \]
By using the inequalities $| \langle e_l^\pi, h_i^\pi \rangle |^2 \leq \Norm{h_i^\pi}^2$, $| \langle g_i^\pi, e_k^\pi \rangle |^2 \leq \Norm{g_i^\pi}^2$, and $|\pi_{k,l}(\xi)|\leq 1$, one can easily show that 
\[ \sum_{i=1}^{m_\pi} | \langle e_l^\pi, h_i^\pi \rangle \langle g_i^\pi, e_k^\pi \rangle \overline{\pi_{k,l}(\xi)} | \leq \Norm{h_\pi} \Norm{g_\pi} < \infty \qquad (\pi \in \hat{K};\ k,l=1,\dotsc,d_\pi). \]
Thus, we can reorder the sum above to write
\[ \langle g, \rho(\xi) h \rangle = \sum_{\pi \in \hat{K}} \sum_{k,l=1}^{d_\pi} \left( \frac{1}{\sqrt{d_\pi}} \sum_{i=1}^{m_\pi} \langle e_l^\pi, h_i^\pi \rangle \langle g_i^\pi, e_k^\pi \rangle \right) \sqrt{d_\pi}\, \overline{\pi}_{k,l}(\xi). \]

We want to apply the Peter-Weyl Theorem to conclude that
\begin{equation}\label{eq:multFrm1}
\int_K | \langle g, \rho(\xi) h \rangle |^2 d\xi = \sum_{\pi \in \hat{K}} \frac{1}{d_\pi} \sum_{k,l=1}^{d_\pi} \left| \sum_{i=1}^{m_\pi} \langle e_l^\pi, h_i^\pi \rangle \langle g_i^\pi, e_k^\pi \rangle \right|^2.
\end{equation}
To justify \eqref{eq:multFrm1}, it suffices to prove the sum on the right is finite. To see this is the case, first observe that for $\pi \in \hat{K}$,
\[ \sum_{i=1}^{m_\pi} \langle e_l^\pi, h_i^\pi \rangle \langle g_i^\pi, e_k^\pi \rangle = \langle (T_\pi h_\pi) e_l^\pi, (T_\pi g_\pi) e_k^\pi \rangle \qquad (k,l=1,\dotsc,d_\pi). \]
Denoting $\Norm{\cdot}_{\text{op}}$ for the operator norm, we have
\[ \frac{1}{d_\pi} \sum_{k,l=1}^{d_\pi} | \langle (T_\pi h_\pi) e_l^\pi, (T_\pi g_\pi) e_k^\pi \rangle |^2 = \frac{1}{d_\pi} \sum_{l=1}^{d_\pi} \Norm{ (T_\pi g_\pi)^* (T_\pi h_\pi) e_l^\pi }^2 \leq \Norm{ (T_\pi g_\pi)^* (T_\pi h_\pi) }_{\text{op}}^2 \]
\[  \leq \Norm{ T_\pi g_\pi }_{\text{op}}^2 \Norm{ T_\pi h_\pi }_{\text{op}}^2 \leq \Norm{ T_\pi g_\pi }_{\HS}^2 \Norm{ T_\pi h_\pi }_{\HS}^2. \]
Since $\Norm{g}^2 = \sum_{\pi \in \hat{K}} \Norm{ T_\pi g_\pi }_{\HS}^2$, there is some $M > 0$ such that $\Norm{ T_\pi g_\pi }_{\HS}^2 \leq M$ for all $\pi \in \hat{K}$. Hence,
\[ \sum_{\pi \in \hat{K}} \frac{1}{d_\pi} \sum_{k,l=1}^{d_\pi} \left| \sum_{i=1}^{m_\pi} \langle e_l^\pi, h_i^\pi \rangle \langle g_i^\pi, e_k^\pi \rangle \right|^2 \leq \sum_{\pi \in \hat{K}} \Norm{ T_\pi g_\pi }_{\HS}^2 \Norm{ T_\pi h_\pi }_{\HS}^2 \leq M \Norm{h}^2 < \infty. \]
This proves \eqref{eq:multFrm1}.

We continue by refining the expression on the right side of \eqref{eq:multFrm1} even further. For $\pi \in \hat{K}$ and $k \in \{1,\dotsc,d_\pi\}$, we claim that
\begin{equation} \label{eq:multFrm2}
\sum_{l=1}^{d_\pi} \left| \sum_{i=1}^{m_\pi} \langle e_l^\pi, h_i^\pi \rangle \langle g_i^\pi, e_k^\pi \rangle \right|^2 = \Norm{ \sum_{i=1}^{m_\pi} \langle e_k^\pi, g_i^\pi \rangle h_i^\pi }^2. 
\end{equation}
Indeed, we can write
\[ \sum_{l=1}^{d_\pi} \left| \sum_{i=1}^{m_\pi} \langle e_l^\pi, h_i^\pi \rangle \langle g_i^\pi, e_k^\pi \rangle \right|^2 = \sum_{l=1}^{d_\pi} \sum_{i=1}^{m_\pi} \sum_{j=1}^{m_\pi} \langle e_l^\pi, h_i^\pi \rangle \langle g_i^\pi, e_k^\pi \rangle \langle h_j^\pi, e_l^\pi \rangle \langle e_k^\pi, g_j^\pi \rangle \]
\[ = \sum_{i=1}^{m_\pi} \sum_{j=1}^{m_\pi} \langle g_i^\pi, e_k^\pi \rangle \langle e_k^\pi, g_j^\pi \rangle \sum_{l=1}^{d_\pi} \langle e_l^\pi, h_i^\pi \rangle \langle h_j^\pi, e_l^\pi \rangle = \sum_{i=1}^{m_\pi} \sum_{j=1}^{m_\pi} \langle g_i^\pi, e_k^\pi \rangle \langle e_k^\pi, g_j^\pi \rangle \langle h_j^\pi, h_i^\pi \rangle \]
\[ = \sum_{i=1}^{m_\pi} \sum_{j=1}^{m_\pi} \langle \langle e_k^\pi, g_j^\pi \rangle h_j^\pi, \langle e_k^\pi, g_i^\pi \rangle h_i^\pi \rangle. \]
Since $| \langle e_k^\pi, g_i^\pi \rangle |^2 \leq \Norm{g_i^\pi}^2$, one can show that $\sum_{i=1}^{m_\pi} \Norm{ \langle e_k^\pi, g_i^\pi \rangle h_i^\pi } \leq \Norm{ g_\pi} \Norm{ h_\pi } < \infty$. Hence the sum $\sum_{i=1}^{m_\pi} \langle e_k^\pi, g_i^\pi \rangle h_i^\pi$ converges in $\H_\pi$. That means we can move the sums inside the inner product above. This gives \eqref{eq:multFrm2}.

Combining \eqref{eq:multFrm1} with \eqref{eq:multFrm2}, and letting $h$ run through $\A$, we obtain the critical identity
\begin{equation} \label{eq:multFrm3}
\sum_{j\in I} \int_K | \langle g, \rho(\xi) f_j \rangle |^2 d\xi = \sum_{\pi \in \hat{K}} \sum_{k=1}^{d_\pi} \frac{1}{d_\pi} \sum_{j\in I} \Norm{ \sum_{i=1}^{m_\pi} \langle e_k^\pi, g_i^\pi \rangle f_{i,j}^\pi }^2 \qquad (g \in \H_\rho). 
\end{equation}
Meanwhile,
\begin{equation} \label{eq:multFrm4}
\Norm{g}^2 = \sum_{\pi \in \hat{K}} \sum_{k=1}^{d_\pi} \sum_{i=1}^{m_\pi} | \langle e_k^\pi, g_i^\pi \rangle |^2 \qquad (g \in \H_\rho).
\end{equation}

The rest of the proof comes easily. If $g \in V_J$, then $( \langle e_k^\pi, g_i^\pi \rangle )_{i=1}^{m_\pi} \in \ran T_\pi g_\pi \subset J(\pi)$ for every $\pi \in \hat{K}$ and every $k \in \{1,\dotsc,d_\pi\}$. Thus, (ii) implies (i). 

Now assume (i) holds. Fix $\pi \in \hat{K}$, and let $(c_i)_{i=1}^{m_\pi} \in J(\pi)$ be arbitrary. Define $g \in \H_\rho$ by
\[ g_i^\sigma = \begin{cases}
\overline{c_i} e_1^\pi, & \text{ if } \sigma = \pi \\
0, & \text{ if }\sigma \neq \pi
\end{cases}  \qquad (\sigma \in \hat{K},\ 1 \leq i \leq m_\sigma ). \]
Then 
\[ \Norm{g}^2 = \sum_{i=1}^{m_\pi} |c_i|^2, \]
while \eqref{eq:multFrm3} gives
\[ \sum_{j\in I} \int_K | \langle g, \rho(\xi) f_j \rangle |^2 d\xi = \frac{1}{d_\pi} \sum_{j\in I} \Norm{ \sum_{i=1}^{m_\pi} c_i f_{i,j}^\pi }^2. \]
Since $\ran (Tg)(\sigma) \subset J(\sigma)$ for each $\sigma \in \hat{K}$, (i) applies to tell us that
\[ A \sum_{i=1}^{m_\pi} | c_i |^2 \leq \frac{1}{d_\pi} \sum_{j\in I} \Norm{ \sum_{i=1}^{m_\pi} c_i f_{i,j}^\pi }^2 \leq B \sum_{i=1}^{m_\pi} |c_i|^2. \]
This is (ii).
\end{proof}

\medskip

\begin{cor} \label{cor:rowSum}
If $E(\A)$ is a continuous frame for $S(\A)$, then every row of $\A(\pi)$ belongs to $\H_\pi^{\oplus I}$, for every $\pi \in \hat{K}$.
\end{cor}

\begin{proof}
Fix $\pi \in \hat{K}$. Let $i_0 \in \{1,\dotsc, m_\pi\}$ when $m_\pi < \infty$ and $i_0 \in \N$ when $m_\pi = \infty$. Denote $\delta_{i_0} \in \ell^2_{m_\pi}$ for the vector with a $1$ in the $i_0$-th coordinate and $0$ in all others, and let $J$ be the range function given by
\[ J(\sigma) = \begin{cases}
\spn\{\delta_{i_0}\}, & \text{if }\sigma = \pi \\
\{0\}, & \text{if } \sigma \neq \pi.
\end{cases} \]
Then $V_J$ is the $i_0$-th summand of $\H_\pi^{\oplus m_\pi} \subset \H_\rho$. Let $P\colon S(\A) \to V_J$ be the restriction to $S(\A)$ of the orthogonal projection $\H_\rho \to V_J$. Since $V_J$ is an invariant subspace of $\H_\rho$, $P$ commutes with $\rho(\xi)$ for every $\xi \in K$, and the range of $P$ is an invariant subspace of $V_J$.

Since $V_J$ is irreducible, one of two things must happen: either the range of $P$ is zero, or it is all of $V_J$. In the former case, we have $f_{i_0,j}^\pi = 0$ for all $j \in I$, so that the $i_0$-th row of $\A(\pi)$ equals $0 \in \H_\pi^{\oplus I}$. In the latter case, $\{P \rho(\xi) f_j \}_{j \in I, \xi \in K} = \{ \rho(\xi) P f_j \}_{j\in I, \xi \in K}$ is a continuous frame for $V_J$. Say the upper bound is $B > 0$. Applying Theorem \ref{thm:multFrm} with $\delta_{i_0}$ in place of $(c_i)_{i=1}^{m_\pi}$, we find that
\[ \sum_{j\in I} \Norm{ f_{i_0,j}^\pi }^2 \leq B d_\pi < \infty. \]
Thus, $(f_{i_0,j})_{j\in I} \in \H_\pi^{\oplus I}$.
\end{proof}

The work above assumes that $\A$ is countable and that our continuous frames $\{ \rho(\xi) f_j\}_{j \in I, \xi \in K}$ are taken over the measure space $I \times K$, where $I$ is equipped with counting measure. We have imposed this assumption only for the sake of clarity. Our arguments work just as well (with obvious modifications) if we replace $I$ with a $\sigma$-finite measure space $(X,\mu)$, and allow $\A = \{f_x\}_{x\in X}$ to be a possibly uncountable family of vectors. We have to assume, however, that the mapping $x \mapsto f_x$ is weakly measurable from $X$ to $\H_\rho$. We also have to replace the direct sum $\H_\pi^{\oplus I}$ with the direct \emph{integral} $\int_X^\oplus \H_\pi$. (See \cite[\S 7.4]{F} for a definition.) A standard measurability argument, which we omit, proves the mapping $X \times K \to \H_\rho$ given by $(x,\xi) \mapsto \rho(\xi) f_x$ is weakly measurable. We denote $E(\A)$ for this mapping. As in the countable case, we write $f_x = (f_x^\pi)_{\pi \in \hat{K}}$ with $f_x^\pi = (f_{i,x}^\pi)_{i=1}^{m_\pi} \in \H_\pi^{\oplus m_\pi}$ and $f_{i,x}^\pi \in \H_\pi$. Strictly speaking, 
\[ \A(\pi) := (f_{i,x}^\pi)_{1 \leq i \leq m_\pi, x \in X} \]
is no longer a matrix, but a sequence of mappings $X \to \H_\rho$, each given by $x \mapsto f_{i,x}^\pi$ for some $i$. For the sake of analogy, we will still call these mappings \emph{rows} of $\A(\pi)$. Then we have the following results.

\begin{theorem} \label{thm:multFrmInt}
Let $J$ be a range function in $\{ \ell^2_{m_\pi} \}_{\pi \in \hat{K}}$, and assume that $\A \subset V_J$. For constants $A$ and $B$ with $0 < A \leq B < \infty$, the following are equivalent.
\begin{enumerate}[(i)]
\item $E(\A)$ is a continuous frame for $V_J$ with bounds $A,B$. That is,
\[ A \Norm{g}^2 \leq \int_X \int_K | \langle g, \rho(\xi) f_x \rangle |^2\, d\xi\, d x \leq B \Norm{g}^2 \qquad (g \in V_J). \]
\item For every $\pi \in \hat{K}$ and every sequence $(c_i)_{i=1}^{m_\pi} \in J(\pi) \subset \ell^2_{m_\pi}$,
\[ d_\pi A \sum_{i=1}^{m_\pi} |c_i|^2 \leq \int_X \Norm{ \sum_{i=1}^{m_\pi} c_i f_{i,x}^\pi }^2 dx \leq d_\pi B \sum_{i=1}^{m_\pi} |c_i|^2. \]
\end{enumerate}
\end{theorem}

\begin{cor} \label{cor:multFrmInt}
The following are equivalent for constants $A$ and $B$ with $0 < A \leq B < \infty$.
\begin{enumerate}[(i)]
\item $E(\A)$ is a continuous frame for $\H_\rho$ with bounds $A,B$.
\item For every $\pi \in \hat{K}$, the ``rows'' of $\A(\pi)$ belong to $\int_X^\oplus \H_\pi$, where they form a Riesz sequence with bounds $d_\pi A,d_\pi B$.
\end{enumerate}
\end{cor}

We end with an application. Remember that $\mathcal{M}_\pi \subset \H_\rho$ denotes the isotypical component of $\pi\in \hat{K}$ in $\rho$. In terms of our decomposition of $\H_\rho$, $\mathcal{M}_\pi$ is the summand $\H_\pi^{\oplus m_\pi} \subset \H_\rho$. We write $P_\pi$ for orthogonal projection of $\H_\rho$ onto $\mathcal{M}_\pi$. The result below generalizes Corollary \ref{cor:compFrm} for frames with multiple generators. It is a trivial consequence of Corollary \ref{cor:multFrmInt}.

\begin{cor} \label{cor:multCompFrm}
Let $\A$ be as described in the paragraph above Theorem \ref{thm:multFrmInt}. The following are equivalent for constants $A$ and $B$ with $0 < A \leq B < \infty$.
\begin{enumerate}[(i)]
\item $E(\A)$ is a continuous frame for $\H_\rho$ with bounds $A,B$.
\item For each $\pi \in \hat{K}$, the mapping $X \times K \to \mathcal{M}_\pi$ given by $(x,\xi) \mapsto \rho(\xi) P_\pi f_x$ is a continuous frame for $\mathcal{M}_\pi$ with bounds $A,B$.
\end{enumerate}
\end{cor}

\medskip

\section{Acknowledgements}
The author thanks the following people for insightful comments and conversations: Marcin Bownik, Eusebio Gardella, Eugenio Hern{\'a}ndez, John Jasper, Peter Luthy, Azita Mayeli, Chris Phillips, Ken Ross, and Shayne Waldron. Extra thanks go to Marcin Bownik and Ken Ross, who both read the manuscript and gave helpful suggestions. This research was supported in part by NSF grant DMS-1265711.


\bibliographystyle{abbrv}
\bibliography{comGrp}

\def\cprime{$'$}
\begin{thebibliography}{10}

\bibitem{ACHKM}
A.~Aldroubi, C.~Cabrelli, C.~Heil, K.~Kornelson, and U.~Molter.
\newblock Invariance of a shift-invariant space.
\newblock {\em J. Fourier Anal. Appl.}, 16(1):60--75, 2010.

\bibitem{AAG}
S.~T. Ali, J.-P. Antoine, and J.-P. Gazeau.
\newblock Continuous frames in {H}ilbert space.
\newblock {\em Ann. Physics}, 222(1):1--37, 1993.

\bibitem{ACP}
M.~Anastasio, C.~Cabrelli, and V.~Paternostro.
\newblock Extra invariance of shift-invariant spaces on {LCA} groups.
\newblock {\em J. Math. Anal. Appl.}, 370(2):530--537, 2010.

\bibitem{ACP2}
M.~Anastasio, C.~Cabrelli, and V.~Paternostro.
\newblock Invariance of a shift-invariant space in several variables.
\newblock {\em Complex Anal. Oper. Theory}, 5(4):1031--1050, 2011.

\bibitem{Ba}
R.~Balan.
\newblock {\em A study of {W}eyl-{H}eisenberg and wavelet frames}.
\newblock Princeton University, Princeton, NJ, 1998.
\newblock Ph.D. Thesis.

\bibitem{BHM}
D.~Barbieri, E.~Hern{\'a}ndez, and A.~Mayeli.
\newblock Bracket map for the {H}eisenberg group and the characterization of
  cyclic subspaces.
\newblock {\em Appl. Comput. Harmon. Anal.}, 37(2):218--234, 2014.

\bibitem{BHP}
D.~Barbieri, E.~Hern{\'a}ndez, and J.~Parcet.
\newblock {R}iesz and frame systems generated by unitary actions of discrete
  groups.
\newblock {\em Appl. Comput. Harmon. Anal.}, 2015.
\newblock In press.

\bibitem{BHP3}
D.~Barbieri, E.~Hern{\'a}ndez, and V.~Paternostro.
\newblock Noncommutative shift-invariant spaces.
\newblock Preprint, 2015.

\bibitem{BHP2}
D.~Barbieri, E.~Hern{\'a}ndez, and V.~Paternostro.
\newblock The {Z}ak transform and the structure of spaces invariant by the
  action of an {LCA} group.
\newblock {\em J. Funct. Anal.}, 269(5):1327--1358, 2015.

\bibitem{B}
M.~Bownik.
\newblock The structure of shift-invariant subspaces of {$L^2({\bf R}^n)$}.
\newblock {\em J. Funct. Anal.}, 177(2):282--309, 2000.

\bibitem{B2}
M.~Bownik.
\newblock The structure of shift-modulation invariant spaces: the rational
  case.
\newblock {\em J. Funct. Anal.}, 244(1):172--219, 2007.

\bibitem{BR}
M.~Bownik and K.~A. Ross.
\newblock The structure of translation-invariant spaces on locally compact
  abelian groups.
\newblock {\em J. Fourier Anal. Appl.}, 21(4):849--884, 2015.

\bibitem{CP}
C.~Cabrelli and V.~Paternostro.
\newblock Shift-invariant spaces on {LCA} groups.
\newblock {\em J. Funct. Anal.}, 258(6):2034--2059, 2010.

\bibitem{CP2}
C.~Cabrelli and V.~Paternostro.
\newblock Shift-modulation invariant spaces on {LCA} groups.
\newblock {\em Studia Math.}, 211(1):1--19, 2012.

\bibitem{CW}
T.-Y. Chien and S.~Waldron.
\newblock The projective symmetry group of a finite frame.
\newblock Preprint, 2014.

\bibitem{CW2}
T.-Y. Chien and S.~Waldron.
\newblock A characterisation of projective unitary equivalence of finite
  frames.
\newblock Preprint, 2015.

\bibitem{CMO}
B.~Currey, A.~Mayeli, and V.~Oussa.
\newblock Characterization of shift-invariant spaces on a class of nilpotent
  {L}ie groups with applications.
\newblock {\em J. Fourier Anal. Appl.}, 20(2):384--400, 2014.

\bibitem{DGM}
I.~Daubechies, A.~Grossmann, and Y.~Meyer.
\newblock Painless nonorthogonal expansions.
\newblock {\em J. Math. Phys.}, 27(5):1271--1283, 1986.

\bibitem{BDR}
C.~de~Boor, R.~A. DeVore, and A.~Ron.
\newblock The structure of finitely generated shift-invariant spaces in
  {$L_2({\bf R}^d)$}.
\newblock {\em J. Funct. Anal.}, 119(1):37--78, 1994.

\bibitem{D}
J.~Dixmier.
\newblock {\em {$C\sp*$}-algebras}.
\newblock North-Holland Publishing Co., Amsterdam-New York-Oxford, 1977.
\newblock Translated from the French by Francis Jellett, North-Holland
  Mathematical Library, Vol. 15.

\bibitem{FG}
J.~Feldman and F.~P. Greenleaf.
\newblock Existence of {B}orel transversals in groups.
\newblock {\em Pacific J. Math.}, 25:455--461, 1968.

\bibitem{F}
G.~B. Folland.
\newblock {\em A course in abstract harmonic analysis}.
\newblock Studies in Advanced Mathematics. CRC Press, Boca Raton, FL, 1995.

\bibitem{GH}
J.-P. Gabardo and D.~Han.
\newblock Frames associated with measurable spaces.
\newblock {\em Adv. Comput. Math.}, 18(2-4):127--147, 2003.

\bibitem{G}
I.~M. Gel{\cprime}fand.
\newblock Expansion in characteristic functions of an equation with periodic
  coefficients.
\newblock {\em Doklady Akad. Nauk SSSR (N.S.)}, 73:1117--1120, 1950.

\bibitem{GM}
F.~Greenleaf and M.~Moskowitz.
\newblock Cyclic vectors for representations of locally compact groups.
\newblock {\em Math. Ann.}, 190:265--288, 1971.

\bibitem{GrMo}
A.~Grossmann and J.~Morlet.
\newblock Decomposition of {H}ardy functions into square integrable wavelets of
  constant shape.
\newblock {\em SIAM J. Math. Anal.}, 15(4):723--736, 1984.

\bibitem{Ha}
D.~Han.
\newblock Classification of finite group-frames and super-frames.
\newblock {\em Canad. Math. Bull.}, 50(1):85--96, 2007.

\bibitem{HL}
D.~Han and D.~R. Larson.
\newblock Frames, bases and group representations.
\newblock {\em Mem. Amer. Math. Soc.}, 147(697):x+94, 2000.

\bibitem{He}
H.~Helson.
\newblock {\em Lectures on invariant subspaces}.
\newblock Academic Press, New York-London, 1964.

\bibitem{He2}
H.~Helson.
\newblock {\em The spectral theorem}, volume 1227 of {\em Lecture Notes in
  Mathematics}.
\newblock Springer-Verlag, Berlin, 1986.

\bibitem{HSWW2}
E.~Hern{\'a}ndez, H.~{\v{S}}iki{\'c}, G.~L. Weiss, and E.~N. Wilson.
\newblock Cyclic subspaces for unitary representations of {LCA} groups;
  generalized {Z}ak transform.
\newblock {\em Colloq. Math.}, 118(1):313--332, 2010.

\bibitem{HSWW}
E.~Hern{\'a}ndez, H.~{\v{S}}iki{\'c}, G.~L. Weiss, and E.~N. Wilson.
\newblock The {Z}ak transform(s).
\newblock In {\em Wavelets and multiscale analysis}, Appl. Numer. Harmon.
  Anal., pages 151--157. Birkh\"auser/Springer, New York, 2011.

\bibitem{HR2}
E.~Hewitt and K.~A. Ross.
\newblock {\em Abstract harmonic analysis. {V}ol. {II}: {S}tructure and
  analysis for compact groups. {A}nalysis on locally compact {A}belian groups}.
\newblock Die Grundlehren der mathematischen Wissenschaften, Band 152.
  Springer-Verlag, New York-Berlin, 1970.

\bibitem{I}
J.~W. Iverson.
\newblock Subspaces of {$L^2(G)$} invariant under translation by an abelian
  subgroup.
\newblock {\em J. Funct. Anal.}, 269(3):865--913, 2015.

\bibitem{JaLem}
M.~S. Jakobsen and J.~Lemvig.
\newblock Co-compact {G}abor systems on locally compact abelian groups.
\newblock {\em J. Fourier Anal. Appl.}, 2015.
\newblock In press.

\bibitem{JL}
G.~James and M.~Liebeck.
\newblock {\em Representations and characters of groups}.
\newblock Cambridge University Press, New York, second edition, 2001.

\bibitem{JM}
R.~Q. Jia and C.~A. Micchelli.
\newblock Using the refinement equations for the construction of pre-wavelets.
  {II}. {P}owers of two.
\newblock In {\em Curves and surfaces ({C}hamonix-{M}ont-{B}lanc, 1990)}, pages
  209--246. Academic Press, Boston, MA, 1991.

\bibitem{K}
G.~Kaiser.
\newblock {\em A friendly guide to wavelets}.
\newblock Birkh\"auser Boston, Inc., Boston, MA, 1994.

\bibitem{KR}
R.~A. Kamyabi~Gol and R.~Raisi~Tousi.
\newblock A range function approach to shift-invariant spaces on locally
  compact abelian groups.
\newblock {\em Int. J. Wavelets Multiresolut. Inf. Process.}, 8(1):49--59,
  2010.

\bibitem{RK}
R.~Radha and N.~S. Kumar.
\newblock Shift invariant spaces on compact groups.
\newblock {\em Bull. Sci. Math.}, 137(4):485--497, 2013.

\bibitem{RND}
A.~Rahimi, A.~Najati, and Y.~N. Dehghan.
\newblock Continuous frames in {H}ilbert spaces.
\newblock {\em Methods Funct. Anal. Topology}, 12(2):170--182, 2006.

\bibitem{RS}
A.~Ron and Z.~Shen.
\newblock Frames and stable bases for shift-invariant subspaces of {$L_2(\bold
  R^d)$}.
\newblock {\em Canad. J. Math.}, 47(5):1051--1094, 1995.

\bibitem{S}
T.~P. Srinivasan.
\newblock Doubly invariant subspaces.
\newblock {\em Pacific J. Math.}, 14:701--707, 1964.

\bibitem{VW2}
R.~Vale and S.~Waldron.
\newblock Tight frames and their symmetries.
\newblock {\em Constr. Approx.}, 21(1):83--112, 2005.

\bibitem{VW}
R.~Vale and S.~Waldron.
\newblock Tight frames generated by finite nonabelian groups.
\newblock {\em Numer. Algorithms}, 48(1-3):11--27, 2008.

\bibitem{VW3}
R.~Vale and S.~Waldron.
\newblock The construction of ${G}$-Ðinvariant finite tight frames.
\newblock {\em J. Fourier Anal. Appl.}, 2015.
\newblock In press.

\bibitem{Wa}
S.~Waldron.
\newblock Group frames.
\newblock In {\em Finite frames}, Appl. Numer. Harmon. Anal., pages 171--191.
  Birkh\"auser/Springer, New York, 2013.

\bibitem{W2}
A.~Weil.
\newblock {\em L'int\'egration dans les groupes topologiques et ses
  applications}.
\newblock Actual. Sci. Ind., no. 869. Hermann et Cie., Paris, 1940.

\bibitem{W}
A.~Weil.
\newblock Sur certains groupes d'op\'erateurs unitaires.
\newblock {\em Acta Math.}, 111:143--211, 1964.

\end{thebibliography}

\vspace{10 pt}

\end{document}